\documentclass[11pt]{article}
\usepackage[utf8]{inputenc}
\usepackage[a4paper, margin=2.5cm]{geometry}
\usepackage{amsmath,amsfonts,amssymb,amsthm,color,mathtools}
\usepackage[style=numeric-comp,sorting=nyvt,backend=bibtex,doi=false,isbn=false, maxbibnames=7]{biblatex}
\usepackage{todonotes}
\usepackage{enumitem}
\usepackage{authblk}
\usepackage{bm}

\usepackage[width=136mm]{caption}
\usepackage[colorlinks=true,unicode]{hyperref}
\definecolor{ForestGreen}{rgb}{0.15,0.416,0.18}
\definecolor{EgyptBlue}{rgb}{0.063,0.2,0.65}
\hypersetup{
	colorlinks=true,
	linkcolor=EgyptBlue,         
	citecolor=EgyptBlue,          
	urlcolor=ForestGreen          
}

\usepackage[noabbrev,capitalize]{cleveref}
\newtheorem{theorem}{Theorem}
\newtheorem{lemma}[theorem]{Lemma}
\newtheorem{proposition}[theorem]{Proposition}

\theoremstyle{definition}
\newtheorem{definition}[theorem]{Definition}
\newtheorem{remark}[theorem]{Remark}

\numberwithin{theorem}{section}
\numberwithin{equation}{section}
\numberwithin{figure}{section}

\newcommand{\N}{\mathbb{N}}
\newcommand{\R}{\mathbb{R}}

\newcommand{\K}{\mathbb{K}}
\newcommand{\A}{\mathcal{A}}
\newcommand{\cC}{\mathcal{C}}
\newcommand{\cR}{\mathcal{R}}
\newcommand{\cS}{\mathcal{S}}

\DeclareMathOperator{\dist}{dist}
\DeclareMathOperator{\dom}{dom}
\DeclareMathOperator{\id}{id}
\DeclareMathOperator{\signc}{sc}
\DeclareMathOperator{\lip}{Lip}

\newcommand{\intd}{\ensuremath{\,\mathrm{d}}}
\newcommand{\CK}{\mathcal{C}_{\mathbb{K}}}
\newcommand{\CKL}{\mathcal{C}_{\mathbb{K},L_0}}
\newcommand{\eps}{\ensuremath{\varepsilon}}

\makeatletter
\newcommand{\subjclass}[2][2020]{%
  \let\@oldtitle\@title%
  \gdef\@title{\@oldtitle\footnotetext{#1 \emph{MSC Classification:} #2}}%
}
\newcommand{\keywords}[1]{%
  \let\@@oldtitle\@title%
  \gdef\@title{\@@oldtitle\footnotetext{\emph{Keywords:} #1}}%
}
\makeatother
\title{Discrete Lyapunov functional for cyclic systems of differential equations with time-variable or state-dependent delay}



\author[1,2,3]{István Balázs}
\author[1,2,3]{Ábel Garab\footnote{Corresponding author: garab@math.u-szeged.hu.}}

\affil[1]{Bolyai Institute, University of Szeged,\par Aradi vértanúk tere 1, Szeged, H–6720, Hungary}

\affil[2]{HUN-REN--SZTE Analysis and Applications Research Group, \par Bolyai Institute,
University of Szeged}

\affil[3]{National Laboratory for Health Security, University of Szeged, Szeged, Hungary}

\date{}
\begin{document}
	\maketitle
	
\begin{abstract}
\noindent We consider nonautonomous cyclic systems of delay differential equations with variable delay.  Under suitable feedback assumptions, we define an (integer valued) Lyapunov functional related to the number of sign changes of the coordinate functions of solutions. We prove that this functional possesses properties analogous to those established by Mallet-Paret and Sell for the constant delay case and by Krisztin and Arino for the scalar case. We also apply the results to equations with state-dependent delays.
\end{abstract}

\section{Introduction}
Cyclic systems of delay differential equations (DDEs) arise in a variety of models in life sciences \cite{goodwin:65,hastings:tyson:webster:77,mahaffy:80,scheper:et-al:99,wang:jing:chen:05,xiao:cao:08, li:wei:05, chu:liu:magal:ruan:16,ducrot:nadin:14}, computer science \cite{baldi:atiya:94,garab:krisztin:11,guo:huang:07,hsu_yang^3:09,yi:chen:wu:10}, economics \cite{matsumoto:szidarovszky:10} or even in traffic control and automated control \cite{bartha:krisztin:18,avedisov:orosz:15,walther:03-position-control,walther:07}. They also appear in the analysis of traveling wave solutions of some population dynamical models \cite{ducrot:11,ducrot:nadin:14}.

The seminal papers by  Mallet-Paret and Sell \cite{mallet-paret:sell:96-lyapunov, mallet-paret:sell:96-poincare} considered cyclic DDEs of the normalized form
\begin{equation}
\begin{alignedat}{3}
    \dot x^0(t)&=f^0(t,x^0(t),x^1(t)),&\qquad &\\
    \dot x^{i}(t)&=f^{i}(t,x^{i-1}(t),x^i(t),x^{i+1}(t)),&\qquad &i=1,\dots,N-1,\\
    \dot x^N(t)&=f^N(t,x^{N-1}(t),x^N(t),x^0(t-1))&\qquad &
\end{alignedat}\label{eq:x_bidir}
\end{equation}
with the main hypothesis being a feedback assumption in the last component of each function $f^i$ ($0\leq i\leq N$). In \cite{mallet-paret:sell:96-lyapunov}, they developed a discrete Lyapunov functional, acting on the appropriate infinite-dimensional phase space, whose value is based on the number of sign changes on segments of solutions. Moreover, they also established some basic properties of this functional and proved that backwards-bounded solutions have finitely many sign changes on any finite interval.

These results generalized those on similar ODE systems due to Mallet-Paret and Smith \cite{mallet-paret:smith:90} and on scalar DDEs ($N=0$ case) by Mallet-Paret \cite{mallet-paret:88-morse}. Some years later, Krisztin and Arino \cite{krisztin:arino:01} proved analogous results for scalar DDEs with time-variable or state-dependent delay (SDDEs). We note that the basic concept goes back to  My\v{s}kis \cite{myshkis:51}, in the scalar case, and to Smillie  \cite{smillie:84} and Mallet-Paret and Smith \cite{mallet-paret:smith:90}, in case of systems without delays. Similar ideas (using the so-called Matano function \cite{matano:82}) were also fruitfully applied for scalar reaction-diffusion equations by Brunovsk\'{y} and Fiedler \cite{brunovsky:fiedler:86}, and for difference equations by Mallet-Paret and Sell \cite{mallet-paret:sell:03} and by one of us and Pötzsche \cite{garab:poetzsche:19}.

The discrete Lyapunov functional with its properties have proven to be a game changer in the analysis of the global attractor of infinite-dimensional dynamical systems, and in particular of delay differential equations. Without attempting to be exhaustive or chronological, they have been used to obtain: Poincaré--Bendixson-type results for cyclic unidirectional DDEs \cite{mallet-paret:sell:96-poincare}, scalar SDDEs \cite{krisztin:arino:01, kennedy:19}, and even for reaction-diffusion equations \cite{fiedler:mallet-paret:89}; results on the absence/existence and stability of periodic solutions with prescribed oscillation frequency \cite{yi:chen:wu:10,garab:krisztin:11,garab:13,fiedler:nieto:et-al:20,krisztin:vas:11,krisztin:vas:16,ivanov:lani-wayda:20,ivanov:lani-wayda:04,lani-wayda:ivanov:arxiv:24, zhang:guo:17}; detailed description of the global attractor \cite{krisztin:walther:wu:99,krisztin:vas:11,krisztin:vas:16}; absence of superexponential solutions \cite{garab:20,cao:90,cao:92}; existence of Morse decomposition of the global attractor \cite{mallet-paret:88-morse,polner:02,garab:poetzsche:19,garab:20,bartha:garab:krisztin:25}. 

\medskip

Inspired by the broad utility of this technical tool demonstrated above, in this paper we aim to generalize the results of \cite{krisztin:arino:01, mallet-paret:sell:96-lyapunov} on existence and properties (including finiteness) of the discrete Lyapunov functional for cyclic DDEs for the case of time-variable delay, with the simplification that we only consider unidirectional systems (i.e.\ the $i$-th equation is independent of $x^{i-1}$).  We also apply our results to two types of cyclic SDDEs that arise in many applications, namely DDEs with threshold-type delays and another delay implicitly defined by two values.

\medskip    
Let us consider the unidirectional cyclic system of DDEs of the form
	\begin{align}
		\begin{alignedat}{3}
			\dot x^{i}(t)&=f^{i}(t,x^i(t),x^{i+1}(t)),&\qquad& i\in \{0,\dots,N-1\},\\
			\dot x^N(t)&=f^N(t,x^N(t),x^0(\eta(t))) & \qquad&
		\end{alignedat}\label{eq:x_unidir}
	\end{align}
under the standing assumptions that  
	
\begin{enumerate}[label=(H$_\arabic*$)]
		\item \label{f-cont} the functions $f^0,\dots,f^N$ are continuous in their first argument, and $C^1$ in their other arguments;
		\item \label{f-feedback} $f^0,\dots,f^N$ obey the feedback conditions
		\begin{align*}
				vf^i(t,0,v)&>0
             & D_3 f^i(t,0,0)&>0,&i\in \{0,\dots, N-1\},\\
			\delta vf^N(t,0,v)&>0
               &\delta D_3 f^N(t,0,0)&>0&
		\end{align*}
		for all $t$, where $\delta\in\{\pm1\}$ and $D_j$ stands for differentiation w.r.t.\ the $j$-th argument;
\item\label{delay-prop} $\eta$ is continuous and strictly increasing, and the time-variable delay $\tau(t)\coloneq t-\eta(t)$ is strictly positive and bounded by some positive constant $r$ for all $t\in\R$.
\end{enumerate}

The case $N=0$ is also allowed. Then we have $\dot x(t)=f(x(t),x(\eta(t))$ with $x\coloneqq x^0$ and $f\coloneqq f^0$.

Our study is also highly motivated by the fact that any fixed solution of an SDDE can be viewed as a solution of the same DDE with time-variable delay. This observation allows one to apply the results for SDDEs as well, provided that the time-variable delay generated by the solutions satisfies hypothesis \ref{delay-prop}. Moreover, some second- and higher-order DDEs (e.g.\ $\ddot x(t)=f(t,x(\eta(t))$) may also be brought into the form of \eqref{eq:x_unidir} (or \eqref{eq:x_bidir} with possibly time-variable or state-dependent delay), see e.g.\ \cite{koplatadze:kvinikadze:stavroulakis:00,walther:07,braverman:domoshnitsky:stavroulakis:24}. Such equations are close to real-world applications, since control systems in mechanical processes often lead to mathematical models with second-order DDEs.

The remainder of the paper is organized as follows. In the next preparatory section, we set our notation, define solutions, introduce our main tool and topic of this work, the mapping $V$, study its continuity properties, and finally, introduce a related nonautonomous linear DDE. \cref{section:main-results} contains our main results. We establish fundamental properties of $V$ related to monotonicity (\cref{thm:V-decreasing}), strict monotonicity (\cref{thm:V-dropping}), constancy and continuity (\cref{thm:solutions-regularize}) and finiteness (\cref{th:x_0,thm:V-is-finite-on-A}). In \cref{sec:SDDEs} we give examples of two rather general types of SDDE systems for which our results can be applied. In the final section, we discuss some possible generalizations and related open questions.

% First, we prove that $V$ is nondecreasing along solutions (\cref{thm:V-decreasing}), then give conditions under which $V$ strictly decreases (\cref{thm:V-dropping}), as well as those that guarantee that the state of a solution at a given time $t$ is in a subset of the phase space, where $V$ is continuous (\cref{th.

	% \begin{remark}
	% 	\begin{align}
	% 		\begin{aligned}
	% 			\dot x^0(t)&=f^0(t,x^0(t),x^1(t)),\\
	% 			\dot x^{i}(t)&=f^{i}(t,x^{i-1}(t),x^i(t),x^{i+1}(t)),\quad i=1,\dots,N-1,\\
	% 			\dot x^N(t)&=f^N(t,x^{N-1}(t),x^N(t),x^0(\eta(t)))
	% 		\end{aligned}\label{eq:x_bidir}
	% 	\end{align}
	% \end{remark}
	
	%%%%%%%%%%%%%%%%%%%%%%%%%%%%%%%%%%%%%%%%%%%%%%%%%%%%%%%%%%%%%%%%%%%%%%%%%%%%%%%%%%%%%%%%%%%%%%%%%%%%%%
	
	\section{Preliminaries}
	In this section we introduce our basic notation and define our  functions $V^\pm$ whose properties will be discussed in \cref{section:main-results}, consisting in our main results. Here we only study its basic continuity properties. Then we introduce a linear subclass of equation \eqref{eq:x_unidir} and show that, for our purposes, it suffices to restrict our study to this special class.

    \subsection{Notation, phase space, solutions}
	Denote $\K=[-r,0]\cup\{1,\dots N\}$, $\cC_\K=C(\K,\R)$, and
	$$\cC_\K^1=\left\{\varphi\in \cC_\K:\text{ the restriction }\varphi|_{[-r,0]}\text{ is continuously differentiable}\right\}$$
	with norms $\|\cdot\|$ and $\|\cdot\|_1$, respectively, where
	\begin{align*}
		\|\varphi\|&=\sup_{s\in\K}|\varphi(s)|\quad\text{for }\varphi\in\cC_\K\text{, and}\\
		\|\varphi\|_1&=\|\varphi\|+\sup_{s\in[-r,0]}|\dot\varphi(s)|\quad\text{for }\varphi\in\cC_\K^1.
	\end{align*}
	
	If $x=(x^0,\dots,x^N)\colon\R\to\R^{N+1}$ is a continuous function and $t\in\R$, then define the segment $x_t\in\cC_\K$ with
	\begin{equation}\label{eq:segments}
	    x_t(s)=\begin{cases}
		x^0(t+s),&\text{if }s\in[-r,0],\\
		x^s(t),&\text{if }s\in\{1,\dots,N\}.
	\end{cases}
	\end{equation}
	Let $t_0\in\R$, $\omega\in(t_0,\infty]$, $\varphi\in\cC_\K$, $x^0\colon [t_0-r,\omega)\to\R$ be continuous on $[t_0-r,t_0]$ and continuously differentiable on $[t_0,\omega)$, and $x^i\colon [t_0,\omega)\to \R$ be continuously differentiable for $1\le i\le N$. We call $x=(x^0,\dots,x^N)$ a \textit{solution of \eqref{eq:x_unidir} on $[t_0,\omega)$ with initial function $\varphi$ at $t_0$} if \eqref{eq:x_unidir} holds for all $t\in(t_0,\omega)$, and $x_{t_0}=\varphi$. We also occasionally use the notation $x(t)=(x^0(t),\dots,x^N(t))\in \R^{N+1}$ and $|u|=\max_{i\in\{0,\dots,N\}}|u^i|$ for all $u\in\R^{N+1}$.
	
	Let $x^i\colon\R\to\R$ be continuously differentiable, $0\le i\le N$. We call $x=(x^0,\dots,x^N)$ an \textit{entire solution of \eqref{eq:x_unidir}} if it satisfies \eqref{eq:x_unidir} for all $t\in\R$.

    For convenience, if $x$ is a solution of an equation of form \eqref{eq:x_unidir}, we also use the notations
    \begin{equation}\label{N+1-notation}
        x_t(N+1)\coloneqq x^{N+1}(t)\coloneqq  x^0(\eta(t)).
    \end{equation}

    To avoid any ambiguity, we stress here that for solutions of various equations the lower index (such as $x_t$) is solely used to indicate solution segments as defined above. However, subscripts are also used to denote the sequences. On the other hand, superscripts are used, apart from their usual meaning for powers of numbers, to indicate the coordinates of vector functions (see equation \eqref{eq:x_unidir}). It is worth noting here that for a positive integer $k$, we use the notation $\eta^k$ to denote the $k$-th iterate of the function $\eta$ of delayed time. For constant delay $\tau(t)\equiv \tau$, $\eta^k(t)$ simplifies to $t-k\tau$.
    
\subsection{Function \texorpdfstring{$\bm{V}$}{V} and its (semi-)continuity}
	For $\varphi\in\cC_\K\setminus\{0\}$ and $A\subset\dom\varphi$ introduce the notation
	\begin{align*}
		V^+(\varphi,A)&=\begin{cases}
			\signc(\varphi,A),&\text{if }\signc(\varphi,A)\text{ is even or infinite,}\\
			\signc(\varphi,A)+1,&\text{if }\signc(\varphi,A)\text{ is odd,}
		\end{cases}\\
		V^-(\varphi,A)&=\begin{cases}
			\signc(\varphi,A),&\text{if }\signc(\varphi,A)\text{ is odd or infinite,}\\
			\signc(\varphi,A)+1,&\text{if }\signc(\varphi,A)\text{ is even,}
		\end{cases}
	\end{align*}
	where $\signc(\varphi,A)$ denotes the number of sign changes of  $\varphi$ on the set $A$, that is
	\begin{align}
		\begin{aligned}
			\signc(\varphi,A)=\sup\{k\in \N:{}&\text{there exist }\theta_i\in A,\ 0\le i\le k,\\
			&\text{with }\theta_i<\theta_{i-1}\text{ and }\varphi(\theta_i)\varphi(\theta_{i-1})<0\text{ for }1\le i\le k\},%\\
			%\signc(\varphi)&=\signc(\varphi,\dom\varphi).
		\end{aligned}\label{def:sc}
	\end{align}
	assuming here $\sup\emptyset=0$. Let $\delta\in\{\pm1\}$ be fixed so that $f$ satisfies \ref{f-feedback}. For $N\geq 1$ and $-r\le a<0$,  introduce the sets
    \begin{align*}
		\cS^0&\coloneq\{\varphi\in \cC_\K^1:\text{if }\varphi(0)=0\text{, then }\dot\varphi(0)\varphi(1)>0\},\\
		\cS_a&\coloneq\{\varphi\in \cC_\K^1:\text{if }\varphi(a)=0\text{, then }\delta\varphi(N)\dot\varphi(a)<0\},\\
		\cS^*_a&\coloneq\{\varphi\in \cC_\K^1:\text{if }\varphi(s)=0\text{, for some }s\in[a,0]\text{, then }\dot\varphi(s)\ne0\},\\
		\cS^N_a&\coloneq\{\varphi\in \cC_\K^1:\text{if }\varphi(N)=0\text{, then }\delta\varphi(N-1)\varphi(a)<0\},\\
        \intertext{furthermore, for $N\ge 2$ and $1\le i\le N-1 $ let} 
    	\cS^i&\coloneq\{\varphi\in \cC_\K^1:\text{if }\varphi(i)=0\text{, then }\varphi(i-1)\varphi(i+1)<0\},\\
	\shortintertext{and finally, let} 
	\cR_a&\coloneq \cS_a\cap\left(\bigcap_{i=0}^{N-1} \cS^i\right)\cap \cS^N_a\cap \cS^*_a \subset \CK^1\setminus\{0\}.
    \end{align*}

    For clarity, for $N=0$ we only define the sets $\cS_a$, $\cS^*_a$ and $\cS^0$, where in the latter we replace $\varphi(1)$ by $\varphi(a)$, and set $\cR_a=\cS^0\cap \cS_a \cap \cS^*_a$.
    
    Let
	$$V\coloneq\begin{cases}
		V^+,&\text{if }\delta=1,\\
		V^-,&\text{if }\delta=-1,
	\end{cases}$$
	and denote
	$$\signc(\varphi,a)\coloneq\signc(\varphi,[a,0]\cup\{1,\dots,N\}),\qquad V(\varphi,a)\coloneq V(\varphi,[a,0]\cup\{1,\dots,N\})$$
	for $a\in[-r,0)$. 

    Notice that $\signc(\varphi,a)$ and $V(\varphi, a)$ are finite for all $\varphi\in \cR_a$. Roughly speaking, the set $\cR_a$ contains such members of the phase space, where $\varphi(s)=0$ ($s\in \K$) can only occur, if $\varphi$ changes its sign at $s$. The point of this rather artificially looking definition of the set $\cR_a$ is highlighted by the next lemma, which states that although $V$ is only lower semi-continuous in general, it is continuous on the set $\cR_a$. Moreover, we will later see that if $V(\varphi,\K)$ is finite and $x$ is a solution of \eqref{eq:x_unidir} with $x_{t_0}=\varphi$, then $x_t \in \cR_{-\tau(t)}$ holds for all sufficiently large $t$. 
    %These properties are heavily used when dealing with global attractors. 
	\begin{lemma}\label{le:V-cont}\leavevmode
		\begin{enumerate}[label=(\roman*)]
			\item For every $\varphi\in\cC_\K\setminus\{0\}$, $a\in[-r,0)$, and for every sequence $(\varphi_n)_0^\infty$ in $\cC_\K\setminus\{0\}$ and $(a_n)_0^\infty$ in $[-r,0)$ with $\|\varphi_n - \varphi\| \to 0$ and $a_n\to a$ as $n\to\infty$,
			$$V(\varphi,a)\le\liminf_{n\to\infty}V(\varphi_n,a_n).$$
			\item For every $\varphi\in \cR_a$ and $a\in[-r,0)$, and for every sequence $(\varphi_n)_0^\infty$ in $\cC_\K^1\setminus\{0\}$ and $(a_n)_0^\infty$ in $[-r,0)$ with $\|\varphi_n-\varphi\|_{1}\to0$ and $a_n\to a$ as $n\to\infty$,
			$$V(\varphi,a)=\lim_{n\to\infty}V(\varphi_n,a_n)<\infty.$$
		\end{enumerate}
	\end{lemma}
	\begin{proof}
		(i) The case $\signc(\varphi,a)=0$ is obvious. If $\signc(\varphi,a)=\infty$ then $\varphi=\lim_{n\to\infty}\varphi_n$ and $a=\lim_{n\to\infty}a_n$ easily imply that $\signc(\varphi_n,a_n)\to\infty$ as $n\to\infty$, and thus $V(\varphi,a)=\infty=\liminf_{n\to\infty}V(\varphi_n,a_n)$. 
        
        Assume that $0<\signc(\varphi,a)<\infty$. Choose $k\in\N$ and reals $\theta_0,\theta_1,\dots,\theta_k$ in $\K$ so that $k=\signc(\varphi,a)$, $a<\theta_k<\theta_{k-1}<\dots<\theta_0\le N$ and $\varphi(\theta_i)\varphi(\theta_{i-1})<0$ for all $i\in\{1,\dots,k\}$. Let $\alpha=\min\{|\varphi(\theta_i)|:\ i\in\{0,1,\dots,k\}\}$. Select $n_0\in\N$ such that $\|\varphi_n-\varphi\|<\alpha$ and $a_n<\theta_k$ for all integers $n\ge n_0$. It follows that $\varphi_n(\theta_i)\varphi_n(\theta_{i-1})<0$ for all integers $n\ge n_0$ and $i\in\{1,\dots,k\}$. Consequently, $\signc(\varphi_n,[a_n,0]\cup\{1,\dots,N\})\ge k=\signc(\varphi,a)$ for all integers $n\ge n_0$, and hence $\signc(\varphi,a)\le\liminf_{n\to\infty}\signc(\varphi_n,a_n)$ and $V(\varphi,a)\le\liminf_{n\to\infty}V(\varphi_n,a_n)$ hold.
		
		(ii) Let $\varphi\in \cR_a$. If $\varphi$ has no zeros in $[a,0]\cup\{1,\dots,N\}$, then it is easy to see that for $n$ sufficiently large, $\varphi_n$ has no zeros in $[a_n,0]\cup\{1,\dots,N\}$; it follows that $\signc(\varphi_n,a_n)=\signc(\varphi,a)$ and $V(\varphi_n,a_n)=V(\varphi,a)$.

        Otherwise, the number of zeros of $\varphi$ in $[a,0]\cup\{1,\dots,N\}$ is finite, and they are all simple in $[a,0]$. Suppose that the zeros are given by a strictly decreasing finite sequence $(z_i)_0^k$, $k\in\N$. For $\eps\in(0,1)$ and $i\in\{0,\dots,k\}$, let $U(z_i,\eps)$ denote the open $\eps$-neighborhood of $z_i$ in $\K$. Note that if $z_i\in\{1,\dots,N\}$ for some $i\in\{0,\dots,k\}$ then $U(z_i,\eps)=\{z_i\}$. We can choose $\eps_1>0$ so small that the intervals $U(z_k,\eps_1),U(z_{k-1},\eps_1),\dots,U(z_0,\eps_1)$ are pairwise disjoint, and from the simplicity of the zeros
		\begin{equation*}
			|\dot\varphi(s)|\geq \eps_1\quad\text{for all }s\in[a,0]\cap\bigcup_{i=0}^kU(z_i,\eps_1)
		\end{equation*}
        with appropriate one-sided derivates, where necessary. As the set $$([a,0]\cup\{1,\dots,N\})\setminus\bigcup_{i=0}^kU(z_k,\eps_1)$$ is compact, we can choose $\eps_2>0$ so that
        \begin{equation*}
            |\varphi(s)|\geq \eps_2^2\quad\text{for all }s\in[a,0]\cup\{1,\dots,N\}\setminus\bigcup_{i=0}^kU(z_i,\eps_1).
		\end{equation*}
		Setting
		$$0<\eps \leq \min\{\eps_1, \eps_2\},$$
		we clearly have
		\begin{alignat}{3}
			|\dot\varphi(s)|& \geq  \eps&\quad&\text{for all }s\in[a,0]\cap\bigcup_{i=0}^kU(z_i,\eps).\label{ineq:phi_der}
			\shortintertext{and}
			|\varphi(s)|&\geq \eps^2&\quad&\text{for all }s\in[a,0]\cup\{1,\dots,N\}\setminus\bigcup_{i=0}^kU(z_i,\eps_1).\label{ineq:phi_1}
        \end{alignat}
		Furthermore, by integration we obtain from inequality \eqref{ineq:phi_der} that
		\begin{equation}
            |\varphi(s)|\geq \eps_1\eps\ge\eps^2\quad\text{for all }s\in[a,0]\cap\bigcup_{i=0}^k\left(U(z_i,\eps_1)\setminus U(z_i,\eps)\right).\label{ineq:phi_2}
        \end{equation}
	Inequalities \eqref{ineq:phi_1} and \eqref{ineq:phi_2} imply
	\begin{equation}
	    |\varphi(s)|\geq \eps^2\quad\text{for all }s\in[a,0]\cup\{1,\dots,N\}\setminus\bigcup_{i=0}^kU(z_i,\eps).\label{ineq:phi}
		\end{equation}	
		
        First, consider the case when $a<z_k$ and $z_0<N$. Then for $\eps$ small enough, $a\not\in U(z_k,\eps)$ and $N\not\in U(z_0,\eps)$ hold. Select $n_0\in\N$ such that $\|\varphi_n-\varphi\|_{1}<\eps^2$ and $a_n<z_k-\eps$ for all $n\ge n_0$. Then
		\begin{align*}
			\varphi_n(s)\varphi(s)&>0\quad\text{for all }s\in[a,0]\cup\{1,\dots,N\}\setminus\bigcup_{i=0}^kU(z_i,\eps),\\			\dot\varphi_n(s)\dot\varphi(s)&>0\quad\text{for all }s\in[a,0]\cap\bigcup_{i=0}^kU(z_i,\eps)
		\end{align*}
	hold. This together with inequalities \eqref{ineq:phi_der} and \eqref{ineq:phi} imply $\signc(\varphi_n,a_n)=\signc(\varphi,a)$, so $V(\varphi_n,a_n)=V(\varphi,a)$.
        
        Now assume that $z_k=a$. Then $z_k\ne N$ because of $\varphi\in \cS_a$. A similar argument as in the previous case shows that either $\signc(\varphi_n,a_n)=\signc(\varphi,a)$ or $\signc(\varphi_n,a_n)=\signc(\varphi,a)+1$. From $\delta\varphi(N)\dot\varphi(a)<0$ and the simplicity of the zeros of $\varphi$ it follows that $\signc(\varphi,a)$ is odd if $\delta=1$ and even if $\delta=-1$. Consequently, $V(\varphi_n,a_n)=V(\varphi,a)$. The case $z_0=N$ is analogous.
	\end{proof}

	%%%%%%%%%%%%%%%%%%%%%%%%%%%%%%%%%%%%%%%%%%%%%%%%%%%%%%%%%%%%%%%%%%%%%%%%%%%%%%%%%%%%%%%%%%%%%%%%%%%%%%
	
	\subsection{Related nonautonomous linear system}
	In this part, we show that when studying the properties of $V$ along solutions of equation \eqref{eq:x_unidir}, it suffices to restrict ourselves to a linear subclass. This makes the analysis considerably simpler.
    
	Suppose that $x$ is a solution of \eqref{eq:x_unidir} on an interval $I$. Define
	\begin{equation}\label{def:ab}
		a^i(t)=\int_0^1D_2f^i(t,hx^i(t),x^{i+1}(t))\intd h,\qquad b^i(t)=\int_0^1D_3f^i(t,0,hx^{i+1}(t))\intd h
	\end{equation}
for $0\le i\le N$ (recall notation \eqref{N+1-notation}). Note that
	$$b^i(t)=\begin{dcases}
		\dfrac{f^i(t,0,x^{i+1}(t))}{x^{i+1}(t)}, &\text{if }x^{i+1}(t)\ne0,\\
		 D_3f^i(t,0,0), &\text{if }x^{i+1}(t)=0,
	\end{dcases} $$
and in particular, $b^i(t)>0$ ($0\leq i\leq N-1$) and $\delta b^N(t)>0$ hold for all $t\in I$.
	Then $x$ also solves the system
	\begin{equation}
		\begin{alignedat}{3}
			\dot x^i(t)&=a^i(t)x^i(t)+b^i(t)x^{i+1}(t),&\qquad& i\in \{0,\dots,N-1\} \\
			\dot x^N(t)&=a^N(t)x^N(t)+ b^N(t)x^0(\eta(t))&\qquad&
		\end{alignedat}\label{eq:x}
	\end{equation}
	on $I$.  Setting
	\begin{align}
		\begin{aligned}
			&y^0\colon \eta(I)\cup I\ni t\mapsto \exp\left(-\int_{t_0}^ta^0(s)\,\intd s\right)x^0(t)\in\R,\\
			&y^i\colon I\ni t\mapsto \exp\left(-\int_{t_0}^ta^i(s)\,\intd s\right)x^i(t)\in\R,\qquad i\in \{1,\dots,N\},
		\end{aligned}\label{def:y}
	\end{align}
	we obtain a solution $y=(y^0,\dots,y^N)$ of 
	\begin{align}
		\begin{aligned}
			\dot y^i(t)&=c^i(t)y^{i+1}(t),&\qquad& i\in \{0,\dots,N-1\}\\
			\dot y^N(t)&= c^N(t)y^0(\eta(t)) &\qquad&
		\end{aligned}\label{eq:y}
	\end{align}
	on $I$ with $t_0\in I$ fixed, and
    \begin{align*}
    c^i(t)&\coloneqq b^i(t)\exp\biggl(\int_{t_0}^{t}a^{i+1}(s) \intd s-\int_{t_0}^ta^i(s) \intd s\biggr)>0,\qquad i\in\{0,\dots,N-1\},\\
    \delta c^N(t)&\coloneqq  \delta b^N(t)\exp\biggl(\int_{t_0}^{\eta(t)}a^0(s) \intd s-\int_{t_0}^t a^N(s) \intd s\biggr)>0.
    \end{align*}

    The next lemma reveals why this transformation is beneficial: we may study properties of $V$ along solutions of equation \eqref{eq:y}, instead of \eqref{eq:x_unidir}, while the former equation has a significantly simpler structure and fulfills a strict feedback condition.
	
    \begin{lemma}\label{th:x_t} Let $x$ be a solution of equation \eqref{eq:x_unidir} on an interval $I$ and $y$ be defined by \eqref{def:ab} and \eqref{def:y}. Then the following properties hold:
    \begin{enumerate}[label=(\roman*)]
        \item $y$ solves equation \eqref{eq:y} on $I$;\label{y-is-sol}
        \item $\operatorname{sgn}x_t(s)=\operatorname{sgn}y_t(s)$ for all $t\in I$ and $s\in \K$; \label{sgnx=sgny}
        \item if $x_{t_0}=0\in \CK$ ($y_{t_0}=0$) for some $t_0\in I$, then $x_t=0$ for all $t\in I$; \label{0-is-invariant}
        \item $\signc(x_t,-\tau(t))=\signc(y_t,-\tau(t))$ and therefore $V(x_t,-\tau(t))=V(y_t,-\tau(t))$ hold for all $t\in I$;\label{Vx-equal-Vy}
	\item	if $t\in I$ and $y_t\in \cR_{-\tau(t)}$, then $x_t\in \cR_{-    \tau(t)}$.\label{x-and-y-in-R}
    \end{enumerate}
    \end{lemma}
	\begin{proof} The first four statements are straightforward consequences of the definition of $y$ and the functions $c^i$.
    
	To see \ref{x-and-y-in-R}, let $t\in I$ with $\eta(t)\in I$ and $y_t\in \cR_{-\tau(t)}$ be given. Then we have $y_t\in\cC_\K^1$ and
		\begin{align}
			\dot y^i(s)=\exp\left(-\int_{t_0}^sa^i(\theta)\intd\theta\right)\left(\dot x^i(s)-a^i(s)x^i(s)\right)\label{eq:doty}
		\end{align}
		for $s\in(\eta(t),t]$, $i=0,\dots,N$; at $s=\eta(t)$ the analog of \eqref{eq:doty} for the right derivatives $D^+{y^0}(\eta(t))$ and $D^+{x^0}(\eta(t))$ holds.
		
		Assume that $s\in(\eta(t),t)$ and $x^0(s)=0$. From statement \ref{sgnx=sgny} and $y_t\in \cS^*_{-\tau(t)}$ it follows that $y^0(s)=0$ and $\dot y^0(s)\ne0$. Hence \eqref{eq:doty} implies $\dot x^0(s)\ne 0$ and $x_t\in \cS^*_{-\tau(t)}$.
		
		Now assume that $x^0(t)=0$. Then statement \ref{sgnx=sgny} and $y_t\in \cS^0$ imply $y^0(t)=0$ and $\dot y^0(t)y^1(t)>0$. Hence, \eqref{eq:doty} with $i=0$ and $s=t$ yields $\dot y^0(t)\dot x^0(t)>0$. Using again statement \ref{sgnx=sgny}, it follows that $y^1(t)x^1(t)>0$. Therefore, $\dot x^0(t)x^1(t)>0$ and $x_t\in \cS^0$.
		
		Suppose that $x^i(t)=0$ for some $i\in\{1,\dots,N-1\}$. Again, from statement \ref{sgnx=sgny} and $y_t\in \cS^i$, it follows that $y^i(t)=0$ and $y^{i-1}(t)y^{i+1}(t)<0$. Hence $x^{i-1}(t)x^{i+1}(t)<0$ and $x_t\in \cS^i$.
		
		Now assume that $x^0(\eta(t))=0$. Using statement \ref{sgnx=sgny} and $y_t\in \cS_{-\tau(t)}$ we infer $y^0(\eta(t))=0$ and $\delta y^N(t)\dot y^0(\eta(t))<0$. Thus, from \eqref{eq:doty} it follows that $\dot y^0(\eta(t))\dot x^0(\eta(t))>0$, and then $\delta x^N(t)\dot x^0(\eta(t))<0$, $x_t\in \cS_{-\tau(t)}$.
		
		Finally, if $x^N(t)=0$, then statement \ref{sgnx=sgny} combined with $y_t\in \cS^N_{-\tau(t)}$ implies $y^N(t)=0$ and $\delta y^{N-1}(t)y^0(t-\tau(t))<0$. In view of equation \eqref{eq:doty}, we obtain $\delta x^{N-1}(t)x^0(t-\tau(t))<0$, hence $x_t\in \cS^N_{-\tau(t)}$.
		
		In summary, we obtain $x_t\in \cR_{-\tau(t)}$.
	\end{proof}
	
	%%%%%%%%%%%%%%%%%%%%%%%%%%%%%%%%%%%%%%%%%%%%%%%%%%%%%%%%%%%%%%%%%%%%%%%%%%%%%%%%%%%%%%%%%%%%%%%%%%%%%%
	
	\section{Main results}\label{section:main-results}

This section is devoted to our main results, generalizing analogous results for constant delay \cite{mallet-paret:sell:96-lyapunov} and scalar DDEs with variable delay \cite{krisztin:arino:01}. The proofs combine ideas from these two papers and from the monograph \cite{krisztin:walther:wu:99}.

In \cref{sec:V-basic-prop}, we prove first that $V(x_t,-\tau(t))$ is nonincreasing along solutions of equation \eqref{eq:x_unidir} (see \cref{thm:V-decreasing}). This is followed by a result (\cref{thm:V-dropping}) that provides a criterion when the value of $V$ must drop. \cref{thm:solutions-regularize} yields conditions that assure that a solution segment $x_t$ is in the set $\cR_{-\tau(t)}$, where $V$ is continuous according to \cref{le:V-cont}. Finally, we establish sufficient criteria for the finiteness of $V$ in \cref{sec:V-finite}.

\subsection{Properties of \texorpdfstring{$\bm{V}$}{V}}\label{sec:V-basic-prop}

The following two rather obvious observations will be frequently used in our arguments.

    \begin{remark}\label{re:cut}
		If a real function $\varphi$ is defined on a set $A\subset\R$, and there exists $a\in A$ such that $\varphi(a)\ne0$, then $$\signc(\varphi,A)=\signc(\varphi,A\cap(-\infty,a])+\signc(\varphi,A\cap[a,\infty)).$$
	\end{remark}
    
	\begin{remark}\label{re:theta}
		Assume that $y$ is a solution of equation \eqref{eq:y} on an interval $I$ and $k=\signc(y_t,-\tau(t))$ for some $t\in I$, and suppose that $\theta_i$, $0\le i\le k$ are chosen as in \eqref{def:sc} for $\varphi=y_t$. Then for $\eps>0$ small enough $y_{s}(\theta_i)y_t(\theta_i)>0$, holds for all $i\in \{0,\dots, k\}$ and $s\in (t-\eps,t+\eps)\cap I$.
	\end{remark}

    The next theorem is the reason we refer to $V$ as a (time-variable) discrete Lyapunov functional.
    
	\begin{theorem}\label{thm:V-decreasing}
		Assume that $x$ is a solution of \eqref{eq:x_unidir} on an interval $I$, and $x_t\ne 0$ for some $t\in I$. Then for all $t_1,t_2$ in $I$ with $t_1<t_2$, $V(x_{t_1},-\tau(t_1))\ge V(x_{t_2},-\tau(t_2))$.
	\end{theorem}
	\begin{proof}
		Consider \eqref{def:y}. By \cref{th:x_t}, it is enough to show the statement for $y$ instead of $x$.
		
		We claim that it suffices to show that for all $t \in I$ there exists $\eps^*=\eps^*(t)>0$ such that for all $\eps \in\left[0,\eps^*\right]$ with $t+\eps \in I$,
		\begin{align}
			V(y_t,-\tau(t))\ge V(y_{t+\eps},-\tau(t+\eps))\label{ineq:V}
		\end{align}
		holds. To see this, let $t_1, t_2$ in $I$ with $t_1<t_2$ be given and assume that for every $t \in I$ there is $\eps^*=\eps^*(t)>0$ so that for all $\eps \in[0,\eps^*]$ with $t+\eps \in I$ we have \eqref{ineq:V}. Define 
        $$t^*=\sup \{s\in[t_1,t_2]: V(y_{t_1},-\tau(t_1))\ge V(y_u,-\tau(u))\text{ for all }t_1\le u\le s\}.$$
        Then $t_1<t^*\le t_2$. From the definition of $t^*$ it follows that there is a sequence $(s_n)_0^{\infty}$ in $[t_1,t^*]$ such that $s_n\to t^*$ as $n\to\infty$ and $V(y_{t_1},-\tau(t_1))\ge V(y_{s_n},-\tau(s_n))$ for all $n\in\N$. Clearly, $y_{s_n}\to y_{t^*}$ as $n\to\infty$. Then \cref{le:V-cont} (i) yields $V(y_{t_1},-\tau(t_1))\ge V(y_{t^*},-\tau(t^*))$. If $t^*<t_2$, then there is $\eps^*(t^*) \in(0,t_2-t^*]$ such that $V(y_{t_1},-\tau(t_1))\ge V(y_{t^*},-\tau(t^*))\ge V(y_{t^*+\eps},-\tau(t^*+\eps))$ for all $\eps\in[0,\eps^*(t^*)]$. This contradicts the definition of $t^*$. Consequently, $t^*=t_2$, and thus the claim holds.
		
		The case $V(y_t,-\tau(t))=\infty$ is obvious, so we may assume that $V(y_t,-\tau(t))<\infty$. Let $k=\signc(y_t,-\tau(t))$, and let $n=\signc(y_t,\{0,\dots,N\})\le k$. 
		
		If there exists $i\in\{0,\dots,N\}$ such that $y_t(i)\ne0$, then let us choose the values of $\theta_i$ ($i\in\{0,\dots,n\}$) from \eqref{def:sc} such that
		\begin{equation}
            \begin{aligned}\label{eq:theta-choice}
			\theta_0&\coloneq\max\{i\in\{0,\dots,N\}: y_t(i)\ne0\},\\
			\theta_j&\coloneq\max\{i\in\{0,\dots,\theta_{j-1}-1\}: y_t(i)y_t(\theta_{j-1})<0\},\quad j=1,\dots,n.
            \end{aligned}
		\end{equation}
		and $y_t(i)y_t(\theta_n)\ge0$ holds for all $i\in\{0,\dots,\theta_n\}$, furthermore, if $n<k$, then $\theta_j\in(-\tau(t),0)$ for $j\in\{n+1,\dots,k\}$. 
		
		\begin{figure}[ht]\centering
			\includegraphics[width=0.8\textwidth]{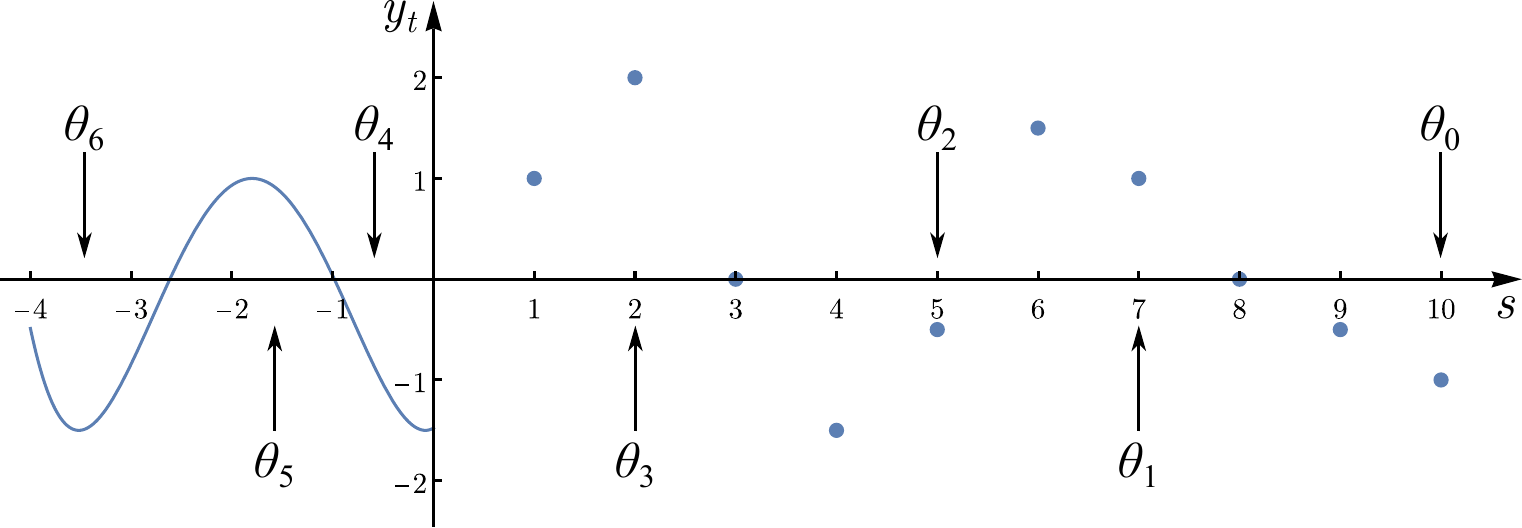}
			\caption{Example for the choice of $\theta_i$, $i=0,\dots,k$. Here $N=10$, $\tau(t)=4$, $k=6$, and $n=3$.\label{fig:thetas}}
		\end{figure}
		
		Hence, $\signc(y_t,[\theta_{j+1},\theta_j]\cap\K)=1$ for $j=0,\dots,k-1$, and by  \cref{re:cut},
		$$\signc(y_t,-\tau(t))=\sum_{j=0}^{k-1}\signc(y_t,[\theta_{j+1},\theta_j]\cap\K)=k.$$
		Let $\eps>0$ be so small, that $y_{t+\eps_1}(\theta_j)$ and $y_t(\theta_j)$ have the same sign for any $t\in I$, $\eps_1\in[0,\eps]$ and $j=0,\dots,k$. Using the continuity of the delay, we can assume that $\theta_k> -\tau(t+\eps)$ also holds.
		       
        There are two cases depending on $\theta_0$.
	\medskip
        
		\noindent \textbf{Case 1.} $\theta_0=N$. Recall that the functions $c^i$ ($i\in \{0,\dots,N-1\}$) and $\delta c^N$ related to system \eqref{eq:y} are all strictly positive. By integrating the $\theta_j-1$-th, \dots, $\theta_{j+1}+1$-th equations in \eqref{eq:y}, we obtain that, for small $\eps>0$ and $0\le j\le n-1$, $y_{t+\eps}(\theta_j)$ and $y_{t+\eps}(l)$ have the same nonzero sign for all  $l\in \{\theta_{j+1},\dots,\theta_j\}$. Similarly, if $\theta_n>0$, then for small $\eps>0$, $y_{t+\eps}(\theta_n)$ and $y_{t+\eps}(l)$ have the same nonzero sign for $l\in \{0,\dots,\theta_n\}$. This together with the continuity of $\eta$ implies that there exists $\eps>0$ small enough that
		$y_t(i)y_{t+\eps}(i)\ge0$ for $i=0,\dots,N$, and $y_t(\theta_j)y_{t+\eps}(\theta_j)>0$ for $j=0,\dots,k$. From the equality
		\begin{align*}
			\signc(y_{t+\eps},-\tau(t+\eps))&=\signc(y_{t+\eps},[-\tau(t+\eps),\theta_k]\cap\K)+\sum_{j=0}^{k-1}\signc(y_{t+\eps},[\theta_{j+1},\theta_j]\cap\K),
		\end{align*}
		as $\signc(y_{t+\eps},[-\tau(t+\eps),\theta_k]\cap\K)=0$ and $\signc(y_{t+\eps},[\theta_{j+1},\theta_j]\cap\K)=1$ for all $j=0,\dots,k-1$, we obtain
		\begin{align*}
			\signc(y_{t+\eps},-\tau(t+\eps))&=k=\signc(y_t,-\tau(t)).
		\end{align*}
		Hence $V(y_{t+\eps},-\tau(t+\eps))=V(y_t,-\tau(t))$.
		\medskip
        
		\noindent\textbf{Case 2.} $\theta_0<N$. As in the previous case, for $\eps>0$ small enough, we have
		\begin{align}
			\signc(y_t,-\tau(t))&=\overbrace{\signc(y_t,[-\tau(t),\theta_0]\cap\K)}^k+\overbrace{\signc(y_t,[\theta_0,N]\cap\K)}^0=k,\nonumber\\
			\signc(y_{t+\eps},-\tau(t+\eps))&=\overbrace{\signc(y_{t+\eps},[-\tau(t+\eps),\theta_0]\cap\K)}^k+\signc(y_{t+\eps},[\theta_0,N]\cap\K). \label{eq:yteps}
		\end{align}
		We claim that $\signc(y_{t+\eps},-\tau(t+\eps))\le k+1$ for sufficiently small $\eps$. We distinguish two cases:
		
		If $\eps_1\in(0,\tau(t))$ can be chosen such that $y_t|_{[-\tau(t),-\tau(t)+\eps_1]}= y^0|_{[t-\tau(t),t-\tau(t)+\eps_1]}\equiv0$, then using continuity of $\eta$ and integrating equations in system \eqref{eq:y} we infer that there exists $\eps>0$ such that 
		$$y_{t+\eps}(N)=\dots=y_{t+\eps}(\theta_0+1)=0.$$
		Hence $\signc(y_{t+\eps},[\theta_0,N]\cap\K)=0$, and $\signc(y_t,-\tau(t))=\signc(y_{t+\eps},-\tau(t+\eps))$ hold, which implies $V(y_t,-\tau(t))=V(y_{t+\eps},-\tau(t+\eps))$.
		
		Otherwise, for all $\eps>0$ small enough,
		$$\int_{t}^{t+\eps}y^0(\eta(s))\,\intd s\ne0$$
		(see the definition of $\theta_k$). It follows again by integration  that $y_{t+\eps}(N),\dots,y_{t+\eps}(\theta_0+1)$ have the same nonzero sign, that is equal to the sign of $y_{t+\eps}(\theta_k)$ if $\delta=1$, and the opposite if $\delta=-1$. Now  $\signc(y_{t+\eps},-\tau(t+\eps))\le k+1$ follows from \eqref{eq:yteps}, and equality holds if and only if $\delta=1$ and $k$ is odd, or $\delta=-1$ and $k$ is even. In both cases $V(y_{t+\eps},-\tau(t+\eps))=V(y_t,-\tau(t))$.
	\end{proof}
	
	Let us denote the iterates of $\eta$ by $\eta^0(t)=t$ and $\eta^n(t)=\eta(\eta^{n-1}(t))$ for $n=1,2,\dots$
	
	\begin{theorem}\label{thm:V-dropping}
		Assume that $x$ is a solution of \eqref{eq:x_unidir} on the interval $I$. Furthermore, suppose $x_t\ne0$ for some $t\in I$ with $\eta^2(t)\in I$ and $x_t(i)=x_t(i+1)=0$ for some $i\in\{0,\dots,N\}$. Then either $V(x_t,-\tau(t))=\infty$ or $V(x_{\eta^2(t)},-\tau(\eta^2(t)))>V(x_t,-\tau(t))$. 
	\end{theorem}
	\begin{proof}
		Consider \eqref{def:y}. By \cref{th:x_t} it is enough to show the statement for $y$ instead of $x$. So suppose the conditions of the theorem are satisfied for $y$. 

        If $V(y_{\eta^2(t)},-\tau(\eta^2(t)))=\infty$, then, in light of \cref{thm:V-decreasing}, there is nothing to prove. So we may suppose for the rest of the proof that $V(y_{\eta^2(t)},-\tau(\eta^2(t)))$ is finite. Then $V(y_{s},-\tau(s))$ is also finite for all $s\in [\eta^2(t),t]$ by \cref{thm:V-decreasing}. 
        
        Let us assume $N\geq 2$ and $V(y_t,-\tau(t))<\infty$. We will frequently use -- without further mention -- that the functions $c^i$ ($0\leq i\leq N-1)$ and $\delta c^N$ related to system \eqref{eq:y} are strictly positive. For the sake of concreteness, let $k=\signc(y_t,-\tau(t))$ and $\theta_0,\dots, \theta_k$ be chosen, as in \eqref{eq:theta-choice}.
        
        There are several cases and subcases to study. %Let $i\in \{0,\dots,N\}$ be the smallest such that $x_t(i)=x_t(i+1)=0$ and ??
        \medskip
        
		\noindent \textbf{Case 1.} If there exists $i \in \{0,\dots, N-2\}$ with $y_t(i)=y_t(i+1)=0$, and $y_t(i+2)\ne 0$, then, for all $\eps\in(0,t-\eta(t))$ sufficiently small we have $y_{t-\eps}(i+2)y_t(i+2)>0$ and $y_{t-\eps}(\theta_l)y_t(\theta_l)>0$ for all $l\in \{0,\dots, k\}$ (cf.~\cref{re:theta}). By integrating the $i+1$-th and the $i$-th equation of \eqref{eq:y} backwards, we obtain $y_{t-\eps}(i+1)y_{t-\eps}(i+2)<0$ and $y_{t-\eps}(i)y_{t-\eps}(i+1)<0$. Hence $\signc(y_{t-\eps},-\tau(t-\eps))\ge\signc(y_t,-\tau(t))+2$ and, by  \cref{thm:V-decreasing},
        \begin{align*}
            V(y_{\eta^2(t)},-\tau(\eta^2(t)))&\ge V(y_{\eta(t)},-\tau(\eta(t)))\ge V(y_{t-\eps},-\tau(t-\eps))\\
            &\ge V(y_t,-\tau(t))+2>V(y_t,-\tau(t)).
        \end{align*}

        \medskip
	\noindent \textbf{Case 2.} If $i=N-1$, $y_t(N-1)=y_t(N)=0$, and $y_t(N+1)=y_t(-\tau(t))\ne 0$, then using hypothesis \ref{delay-prop} and integrating equation \eqref{eq:y} backwards, there exists $\eps>0$ small enough such that $\delta y_{t-\eps}(N)y_{t-\eps}(-\tau(t-\eps))<0$ and $y_{t-\eps}(N-1)y_{t-\eps}(N)<0$. On the other hand, in view of \cref{re:theta}, it can also be achieved that $y_{t-\eps}(\theta_l)y_t(\theta_l)>0$ holds for all $l\in \{0,\dots, k\}$. Thus $\signc(y_{t-\eps},-\tau(t-\eps))\ge\signc(y_t,-\tau(t))+1$. If $\delta=1$, then $\signc(y_{t-\eps},-\tau(t-\eps))$ is odd, and if $\delta=-1$, then $\signc(y_{t-\eps},-\tau(t-\eps))$ is even. Exploiting \cref{thm:V-decreasing} and the definition of $V$ we obtain in both cases
        $$V(y_{\eta^2(t)},-\tau(\eta^2(t)))\ge V(y_{t-\eps},-\tau(t-\eps))>V(y_t,-\tau(t)),$$
    as desired.
		
	\medskip
	\noindent \textbf{Case	3.} Assume that $i=N$, $y_t(N)=y_t(N+1)=y_{\eta(t)}(0)=y_t(-\tau(t))=0$. This case has two subcases.
		
		% 3.1. If $y^0$ has infinitely many sign changes on $[\eta^2(t),\eta(t)]$, then 
  %       $$V(y_{\eta^2(t)},-\tau(\eta^2(t)))\ge V(y_{\eta(t)},-\tau(\eta(t)))=\infty\ge V(y_t,-\tau(t)),$$
  %       and the statement clearly holds. In the following cases we may assume that $V(y(s),-\tau(s))<\infty$ for all $s\ge\eta(t)$.
		
		3.1 Assume that $y^0(s)=0$ for all $s\in[\eta^2(t),\eta(t)]$. By system \eqref{eq:y}, $y^1(s)=y^2(s)=\dots=y^N(s)=0$ follows for all $s\in[\eta^2(t),\eta(t)]$, and hence $y_{\eta(t)}=0$, but then \cref{th:x_t}\,\ref{0-is-invariant} implies $y_t=0$, contradicting the assumptions of the theorem, so this case can be excluded.
		
		3.2. Otherwise, there exist nonnegative $\eps_1$ and positive $\eps_2$ such that $\eps_1+\eps_2<\tau(\eta(t))$, and
        \begin{gather}\label{eq:y0-const-0}
            y^0(s)=0\quad \text{for } s\in[\eta(t)-\eps_1,\eta(t)],\\
		    \int_{\eta(t)-\eps_1-\eps}^{\eta(t)-\eps_1}y^0(s)\,\intd s\ne0\quad \text{and}\quad 
            \signc(y^0, [\eta(t)-\eps_1-\eps,\eta(t)-\eps_1])=0\quad \forall\eps\in[0,\eps_2].\label{eq:int_y_0}
	\end{gather}
		Let $j\in\{0,\dots,N\}$ be such that $y_{\eta(t)}(0)=\dots=y_{\eta(t)}(j)=0$ and $y_{\eta(t)}(j+1)\ne0$, if there is such. We distinguish three further subcases: $j=0$, $j\in \{1,\dots,N\}$, or there is no such $j$.
		
		3.2.1. Case $j=0$. Then $y_{\eta(t)}(1)=y^1(\eta(t))\ne0$, and by system \eqref{eq:y}, $\eps_1=0$ follows, moreover $y^0$ has a simple zero and a sign change at $\eta(t)$. Taking into account hypothesis \ref{delay-prop}, condition \eqref{eq:int_y_0}, and \cref{re:theta}, $\signc(y_{t-\eps},-\tau(t-\eps))>\signc(y_t,-\tau(t))$ holds for small positive $\eps$. Recall that $y_t(N)=0$, so analogously to case 2, if $\delta=1$, then by integrating equation \eqref{eq:y} backwards, we infer $y_{t-\eps}(N)y_{t-\eps}(-\tau(t-\eps))<0$, or in other words, $\signc(y_{t-\eps},-\tau(t-\eps))$ is odd; similarly, if $\delta=-1$, then $\signc(y_{t-\eps},-\tau(t-\eps))$ is even. In both cases 
        we have
        $$V(y_{\eta^2(t)},-\tau(\eta^2(t)))\ge V(y_{t-\eps},-\tau(t-\eps))>V(y_t,-\tau(t)),$$
		for sufficiently small $\eps>0$, hence and the statement holds.
        
		3.2.2. Case $j\in\{1,\dots,N\}$. Then $y_{\eta(t)}(0)=y_{\eta(t)}(1)=\dots=y_{\eta(t)}(j)=0$ and $y_{\eta(t)}(j+1)\ne0$. 
        Apply case 1 (if $j\leq N-1)$), or case 2 (if $j=N$) with $\eta(t)$ in place of $t$. By virtue of \cref{thm:V-decreasing} we obtain that for $\eps>0$ small enough one has
        \begin{equation*}
			V(y_{\eta^2(t)},-\tau(\eta^2(t)))\ge V(y_{\eta(t)-\eps},-\tau(\eta(t)-\eps)) > V(y_{\eta(t)},-\tau(\eta(t))) \ge V(y_t,-\tau(t)).
		\end{equation*}

  %       Then $y_{\eta(t)}(0)=y_{\eta(t)}(1)=\dots=y_{\eta(t)}(j)=0$ 
  %       Then, similarly to case 1, using continuity of $\eta$, there exists $\eps\in(0,t-\eta^2(t))$ small enough such that $y_{\eta(t-\eps)}(m)y_{\eta(t-\eps)}(m+1)<0$ for $m\in\{0,1,\dots,j-1\}$, and in case $j=N$, then $\delta y_{\eta(t-\eps)}(N)y_{\eta(t-\eps)}(N+1)<0$, $\signc(y_{\eta(t-\eps)},-\tau(\eta(t-\eps)))\ge\signc(y_{\eta(t)},-\tau(\eta(t)))+2$, and then
		% \begin{align*}
		% 	V(y_{\eta^2(t)},-\tau(\eta^2(t)))&\ge V(y_{\eta(t-\eps)},-\tau(\eta(t-\eps)))\ge V(y_{\eta(t)},-\tau(\eta(t)))+2\\
		% 	&>V(y_{\eta(t)},-\tau(\eta(t)))\ge V(y_t,-\tau(t)).
		% \end{align*}
		
		3.2.3. Suppose that there is no such $j$, i.e. 
        $$y_{\eta(t)}(0)=\dots=y_{\eta(t)}(N)=y_{\eta(t)}(N+1)=y_{\eta(t)}(-\tau(\eta(t)))=0$$
        holds, which together with equations \eqref{eq:y} and  \eqref{eq:y0-const-0} yield that 
		$$y^0(s)=y^1(s)=\dots=y^N(s)=y^{N+1}(s)=y^0(\eta(s))=0,\quad \text{for all } s\in[\eta(t)-\eps_1,\eta(t)]$$
		holds as well. By virtue of equations \eqref{eq:y} and \eqref{eq:int_y_0}, inequality $y^l(s)y^{l+1}(s)<0$ holds for all $s\in [\eta(t)-\eps_1-\eps,\eta(t)-\eps_1]$ and $l\in\{0,\dots,N-1\}$, if $\eps>0$ is sufficiently small and thus
        \begin{align*}
		V(y_{\eta^2(t)},-\tau(\eta^2(t)))&\ge V(y_{\eta(t-\eps_1-                  \eps)},-\tau(\eta(t-\eps_1-\eps)))\ge V(y_{\eta(t-       \eps_1)}\eta(t-\eps_1))+2\\
            &>V(y_{\eta(t-\eps_1)},-\tau(\eta(t-\eps_1)))\ge V(y_t,-\tau(t)).
	\end{align*}
    This completes the proof for $N\ge 2$.
\medskip

    If $N=1$, then we can exclude case 1, otherwise the proof remains unchanged.
\medskip

    Finally, if $N=0$, then, practically, we only have cases 3.2.1 and 3.2.3. The argument of the proof still holds with straightforward changes in notation. We also note that this case (with negative feedback, i.e.\ $\delta=-1$) is proved in \cite[Lemma 4.2\,(ii)]{krisztin:arino:01}.
% \medskip
% The proof is complete.
    \end{proof}

The next theorem has the consequence that a solution of \eqref{eq:x_unidir} with finitely many sign changes on its initial segment will eventually have its segments in the corresponding $\cR_{-\tau(t)}$ sets, where $V$ is continuous, as shown in \cref{le:V-cont}.

\begin{theorem}\label{thm:solutions-regularize}
    Assume that $x$ is a solution of system \eqref{eq:x_unidir} on $I$. Furthermore, suppose $x_t\ne0$ for some $t\in  I$ with $\eta^3(t)\in I$ and $V(x_{\eta^3(t)},-\tau(\eta^3(t)))=V(x_t,-\tau(t))<\infty$. Then $x_t\in \cR_{-\tau(t)}$ holds.
\end{theorem}

\begin{proof}
    Consider \eqref{def:y}. By \cref{th:x_t}, it is enough to show the statement for $y$ instead of $x$. Some of the following steps can be skipped if $N=1$ or $N=0$ (see the definition of~$\cR_a$). In the latter case, other minor straightforward modifications have to be made. We note that this case (with $\delta=-1$) has already been proved in \cite[Lemma~4.2\,(iii)]{krisztin:arino:01}. 

        Note that from $\eta^3(t)\in I$, and equation \eqref{eq:y} it follows that $y_t\in \CK^1$.
        
		Indirectly, assume that $y_t\not\in \cS^*_{-\tau(t)}$. Then there exists $s\in[-\tau(t),0]$ such that $y_t(s)=\dot y_t(s)=0$ or, equivalently, $y_{t+s}(0)=\dot y_{t+s}(0)=0$. Then, by system \eqref{eq:y}, $y_{t+s}(1)=0$, and, using \cref{thm:V-dropping} we infer the inequalities
        $$V(y_{\eta^3(t)},-\tau(\eta^3(t)))\ge V(y_{\eta^2(t+s)},-\tau(\eta^2(t+s)))>V(y_{t+s},-\tau(t+s))\ge V(y_t,-\tau(t)),$$
        a contradiction, hence $y_t\in \cS^*_{-\tau(t)}$.
		
		Now assume that $y_t(0)=0$. From the previous argument, we obtain $\dot y_t(0)\ne0$, and by system \eqref{eq:y}, $\dot y_t(0)y_t(1)>0$, we arrive at $y_t\in \cS^0$.
		
		Let $1\le i\le N-1$, $y_t(i)=0$ and $y_t(i-1)y_t(i+1)\ge0$. If equality holds, then from \cref{thm:V-dropping} we get a contradiction. Otherwise, by system \eqref{eq:y}, $\dot y_t(i)y_t(i+1)>0$, implying $y_{t-\eps}(i-1)y_{t-\eps}(i)<0$ and $y_{t-\eps}(i)y_{t-\eps}(i+1)<0$ for $\eps\in(0,t-\eta^3(t))$ small enough. Applying  \cref{re:theta}, we obtain $\signc(y_{t-\eps},-\tau(t-\eps))\ge\signc(y_t,-\tau(t))+2$, and
        \begin{align}
            V(y_{\eta^3(t)},-\tau(\eta^3(t)))\ge V(y_{t-\eps},-\tau(t-\eps))\ge V(y_t,-\tau(t))+2>V(y_t,-\tau(t)),\label{eq:V}
        \end{align}
        a contradiction, hence $y_t\in \cS^i$.
		
		Now assume indirectly that $y_t\notin \cS^N_{-\tau(t)}$, i.e.\ $y_t(N)=0$, and $\delta y_t(N-1)y_t(N+1)\ge0$. Equality contradicts \cref{thm:V-dropping}, otherwise $\delta y_t(N-1)y_t(-\tau(t))>0$ holds. By system \eqref{eq:y}, we have $\delta\dot y_t(N)y_t(-\tau(t))>0$, implying $\delta y_{t-\eps}(N)y_{t-\eps}(-\tau(t-\eps))<0$ and $y_{t-\eps}(N-1)y_{t-\eps}(N)<0$ for $\eps\in(0,t-\eta^3(t))$ small enough. Applying \cref{re:theta}, we obtain $\signc(y_{t-\eps},-\tau(t-\eps))>\signc(y_t,-\tau(t))$, and if $\delta=1$ then $\signc(y_{t-\eps},-\tau(t-\eps))$ is odd, if $\delta=-1$ then $\signc(y_{t-\eps},-\tau(t-\eps))$ is even. In both cases, inequality \eqref{eq:V} holds, which is a contradiction, hence $y_t\in \cS^N_{-\tau(t)}$.
		
		Finally, assume to the contrary that $y_t(-\tau(t))=0$ and $\delta y_t(N)\dot y_t(-\tau(t))\ge0$. In light of \cref{thm:V-dropping}, equality cannot hold here. Hence, $y^0$ has a sign change at $\eta(t)$, and so \linebreak $\signc(y_{t-\eps},-\tau(t-\eps))>\signc(y_t,-\tau(t))$ for $\eps\in(0,t-\eta^3(t))$ small enough. Similarly to the previous argument, if $\delta=1$, then $\signc(y_{t-\eps},-\tau(t-\eps))$ is odd, if $\delta=-1$, then $\signc(y_{t-\eps},-\tau(t-\eps))$ is even. In both cases, inequality \eqref{eq:V} holds, which is a contradiction, hence $y_t\in \cS_{-\tau(t)}$.
		
		In summary, we have $y_t\in \cR_{-\tau(t)}$, and the proof is complete.
	\end{proof}
	
	%%%%%%%%%%%%%%%%%%%%%%%%%%%%%%%%%%%%%%%%%%%%%%%%%%%%%%%%%%%%%%%%%%%%%%%%%%%%%%%%%%%%%%%%%%%%%%%%%%%%%%	

	%%%%%%%%%%%%%%%%%%%%%%%%%%%%%%%%%%%%%%%%%%%%%%%%%%%%%%%%%%%%%%%%%%%%%%%%%%%%%%%%%%%%%%%%%%%%%%%%%%%%%%
	
	\subsection{Finiteness of \texorpdfstring{$\bm{V}$}{V}}\label{section:finiteness}\label{sec:V-finite}
	 
     In this subsection we show that under some additional conditions on $f$ and $\eta$, $V$ is finite along backwards-bounded entire solutions, or in other words, any component of such solutions have finitely many sign changes on finite intervals. We have the opposite scenario, when there exists a so-called infinite oscillation point, defined below.
		
	\begin{definition}[\cite{mallet-paret:sell:96-lyapunov}]
		Let $x$ be a solution of system \eqref{eq:x_unidir} on an interval $I$. We say that $\sigma\in I$ is an \textit{infinite oscillation point} (IOP) if there exist a strictly monotone sequence $t_n\to\sigma$ such that $x^0(t_n)x^0(t_{n+1})<0$.
	\end{definition}
	
	\begin{lemma}
		Assume that $x$ is a solution of \eqref{eq:x_unidir} on an interval $I$ and $\sigma$ is an IOP satisfying $\sigma>\inf I$. Then $x(\sigma)=0\in \R^{N+1}$. Moreover, if $\eta(\sigma)\in I$, then $\eta(\sigma)$ is also an IOP.\label{le:iop}
	\end{lemma}
	\begin{proof}
		As $x^i(t)$ has the same sign as $y^i(t)$ defined by \eqref{def:y}, it is enough to prove the statement for the solution of \eqref{eq:y}. Let us assume that $\sigma$ is an IOP of $y$. Then, by continuity, $y^0(\sigma)=0$, and there exists a monotone sequence $(t^0_n)_{n=0}^\infty$ such that $t^0_n\to\sigma$ and $y^0(t^0_n)y^0(t^0_{n+1})<0$. Let us assume that $0\le i<N$ and there exists a monotone sequence $(t^i_n)_{n=0}^\infty$ such that $t^i_n\to\sigma$ and $y^i(t^i_n)y^i(t^i_{n+1})<0$. By the mean value theorem there exists a monotone sequence $(t^{i+1}_n)_{n=0}^\infty$ such that $t^{i+1}_n\in(t^i_n,t^i_{n+1})$, $t^{i+1}_n\to\sigma$ and $\dot y^i(t^{i+1}_n)\dot y^i(t^{i+1}_{n+1})<0$ which together with \eqref{eq:y} imply $y^{i+1}(t^{i+1}_n)y^{i+1}(t^{i+1}_{n+1})<0$ and $y^{i+1}(\sigma)=0$. By induction, we obtain that for all $i\in\{0,\dots, N\}$ there exists a monotone sequence $(t^i_n)_{n=0}^\infty$ such that $t^i_n\to\sigma$, $y^i(t^i_n)y^i(t^i_{n+1})<0$ and $y^i(\sigma)=0$. Again, by the mean value theorem, there exists a monotone sequence $(t^{N+1}_n)_{n=0}^\infty$ such that $t^{N+1}_n\to\sigma$ and $\dot y^N(t^{N+1}_n)\dot y^N(t^{N+1}_{n+1})<0$ implying $y^0(\eta(t^{N+1}_n))y^0(\eta(t^{N+1}_{n+1}))<0$ and $y^0(\eta(\sigma))=0$. Using the monotonicity and continuity of~$\eta$, $(\eta(t^{N+1}_n))_{n=0}^\infty$ is also a monotone sequence such that $\eta(t^{N+1}_n)\to\eta(\sigma)$ as $n\to\infty$, hence $\eta(\sigma)$ is also an IOP.
	\end{proof}

    To establish finiteness of $V$, we will assume in addition to \ref{f-cont}--\ref{delay-prop} the following two hypotheses.
	
    \begin{enumerate}[label=(H$_\arabic*$)]
        \setcounter{enumi}{3}
	    \item \label{f-lin-bounded} For any bounded set $B\subset \R^2$ there exists a constant $C=C(B)>0$ such that
		\begin{align*}
			|f^i(t,u,v)|&\le C(|u|+|v|),\quad i=0,\dots,N,\quad\text{for all } t\in \R \text{ and } (u,v)\in B.
		\end{align*}
        \item There exists a positive constant $L$ such that the iterates $\eta^k$ are $L$-Lipschitz continuous for all $k\in \N$. \label{eta-iterates-Lip}
	\end{enumerate}	
    
	\begin{remark}
		If \eqref{eq:x_unidir} is autonomous, hypothesis \ref{f-cont} implies that \ref{f-lin-bounded} holds.
	\end{remark}

    The next theorem combined with \cref{le:iop} proves the finiteness of $V$ along backwards-bounded entire solutions of equation \eqref{eq:x_unidir}.
    
    \begin{theorem}\label{th:x_0}
	    Suppose that hypotheses \emph{\ref{f-cont}--\ref{eta-iterates-Lip}} hold and $x$ is a backwards-bounded entire solution of \eqref{eq:x_unidir}.
        \begin{enumerate}[label=(\roman*)]
        \item If there exists $\sigma\in\R$ such that $x(\eta^k(\sigma))=0\in \R^{N+1}$ for all $k\in\N_0$, then $x(t)=0$ for all $t\in\R$. 
        \item If $x$ is nontrivial, then $V(x_t,-\tau(t))$ is finite for all $t\in \R$.\label{statement:V-finite}
        \end{enumerate}
    \end{theorem}
	
	\begin{proof}
		(i) Let $\sigma$ such that $x^i(\eta^k(\sigma))=0$ for all $k\in\N_0$ and $i=0,\dots,N$, and define for $t\in [0,r]$, $z_k(t)=x(\eta^k(\sigma+t))$. Observe that $z_k(0)=0$ for all $k\in \N_0$. Let $\intd_s$ denote the Stieltjes integration with respect to $s$. In light of \eqref{eq:x_unidir}, hypotheses \ref{f-lin-bounded}--\ref{eta-iterates-Lip}, and the backwards-boundedness of $x$, we obtain
		\begin{align*}
			\left|z_k^i(t)\right|&=\left|x^i(\eta^k(\sigma+t))\right|=\left|\int_0^tf^i(\eta^k(\sigma+s),x^i(\eta^k(\sigma+s)),x^{i+1}(\eta^k(\sigma+s)))\intd_s\eta^k(\sigma+s)\right|\\
			&\le\int_0^t\left|f^i(\eta^k(\sigma+s),z^i_k(s),z^{i+1}_k(s))\right|\left|\intd_s\eta^k(\sigma+s)\right|\le CL\int_0^t(|z^i_k(s)|+|z^{i+1}_k(s)|)\intd s\\
			\intertext{for $i=0,\dots,N-1$, and}
			\left|z_k^N(t)\right|&=\left|x^N(\eta^k(\sigma+t))\right|\\
			&=\left|\int_0^tf^N(\eta^k(\sigma+s),x^N(\eta^k(\sigma+s)),x^0(\eta^{k+1}(\sigma+s)))\intd_s\eta^k(\sigma+s)\right|\\
			&\le\int_0^t\left|f^N(\eta^k(\sigma+s),z^N_k(s),z^0_{k+1}(s))\right|\left|\intd_s\eta^k(\sigma+s)\right|\\
			&\le CL\int_0^t(|z^N_k(s)|+|z^0_{k+1}(s)|)\intd s.
		\end{align*}
        Let $z_k(t)=(z_k^0(t),\dots,z_k^N(t))$ for all $k\in \N_0$  and $Z(t)=(z_0(t),z_1(t),\dots)\in\ell^\infty$.    
        Since $x$ is a backwards-bounded solution, components $z_k$ of $Z$ are uniformly bounded on $[0,r]$, and thanks to hypothesis \ref{f-lin-bounded}, so are their derivatives. Thus $Z$ is continuous, and inequality
		$$\|Z(t)\|_\infty\le 2CL\int_0^t\|Z(s)\|_\infty\intd s$$
		holds for all $t\in [0,r]$. The Gronwall--Belmann lemma implies $Z(t)=0$ on $[0,r]$, which together with \cref{th:x_t}\,\ref{0-is-invariant} yield $x(t)=0$ for all $t\in \R$.
        \medskip

        (ii) Now assume to the contrary that $x$ is a nontrivial backwards-bounded entire solution of system \eqref{eq:x_unidir} with $V(x_t,-\tau(t))=\infty$ for some $t\in\R$. Then $x$ has an IOP $\sigma\in[\eta(t),t]$. By \cref{le:iop}, $x(\eta^k(\sigma))=0$ for all $k\in\N_0$, but then, by the first part, $x(t)=0$ follows for all $t\in\R$, a contradiction that completes the proof.
	\end{proof}

    Before the last theorem of this section, we recall some notions from the theory of nonautonomous dynamical systems. For a thorough overview of the topic, we refer the readers to the monographs \cite{carvalho:langa:robinson:13, anagnostopoulou:poetzsche:rasmussen:23,kloeden:rasmussen:11}.
\medskip

    Let us assume for the rest of this section that the maximal solutions of initial value problems \eqref{eq:x_unidir} with $x_{t_0}=\varphi\in \CK$ are defined on $[t_0, \infty)$. For example, this is the case if the functions $f^i$  ($i\in \{0,\dots,N\}$) are bounded, but more relaxed conditions on dissipativity can also be given.

    \medskip
    
    Then \eqref{eq:x_unidir} induces a \textit{process} $S(\cdot,\cdot)$ on $\CK$, i.e.~a family of continuous maps $\{S(t,s):\ t\ge s\}$ from $\CK$ to itself with the properties
	\begin{itemize}
		\item $S(t,t)=\id$ for all $t\in\R$,
		\item $S(t,s)=S(t,t_1)S(t_1,s)$ for all $t\ge t_1\ge s$,
		\item $(t,s,\varphi)\mapsto S(t,s)\varphi$ is continuous, $t\ge s$, $\varphi\in\cC_\K$,
		\item $S(t,t_0)\varphi=x_t$, where $x_t$ is the unique solution of \eqref{eq:x_unidir} with initial condition $x_{t_0}=\varphi$.
	\end{itemize}
	The family of sets $\{\A(t)\}_{t\in\R}\subset\cC_\K$ is said to be a \textit{global pullback attractor of $S$}, if
	\begin{itemize}
		\item $\A(t)$ is compact for all $t\in\R$,
		\item $\A(\cdot)$ is \textit{invariant under $S$}, i.e. $S(t,s)\A(s)=\A(t)$ holds for all $t\ge s$, and
		\item $\A(t)$ \textit{pullback attracts} all bounded subsets $B\subset\cC_\K$, i.e.
		$$\lim_{s\to-\infty}\dist(S(t,s)B,\A(t))=0,$$
	\end{itemize}
    where $\dist$ denotes here the Hausdoff semidistance. Finally, we say that a family of bounded sets $\{B(t)\}_{t\in\R}$ is \textit{backwards-bounded}, if $\bigcup_{s\le t}B(s)$ is bounded for all $t\in\R$.
	
	\begin{theorem}\label{thm:V-is-finite-on-A}
	Suppose that hypotheses \emph{\ref{f-cont}--\ref{eta-iterates-Lip}} hold and the process (resp.\ semiflow in the autonomous case) generated by \eqref{eq:x_unidir} has a backwards-bounded pullback attractor $\{\A(t)\}_{t\in\R}$ (resp.\ a global attractor $\A$). Then $V(\varphi)<\infty$ for all $\varphi\in\A(t)\setminus \{0\}$, $t\in\R$ (resp.\ $\varphi \in \A$).
	\end{theorem}
	
	\begin{proof} Recall that a backwards-bounded pullback attractor (resp.\ the global attractor) consists of all functions $\varphi\in \CK$ through which there exists a bounded entire solution \cite[Theorem 1.17]{carvalho:langa:robinson:13}, (resp.~\cite[Lemma 2.18]{raugel:02}). In other words, it consists of solution segments $x_t$ of entire solutions. The statement now follows from \cref{th:x_0}\,\ref{statement:V-finite}.
	\end{proof}

As a corollary of \cref{thm:V-is-finite-on-A,thm:solutions-regularize,thm:V-decreasing}, one obtains that for any bounded entire solution $x$, the segments $x_t$ are all in $\cR_{-\tau(t)}$ for all sufficiently large $|t|$. This fact comes very handy when studying the structure and properties of the attractor.

    \section{Cyclic systems with state-dependent delay}\label{sec:SDDEs}

    Let us consider the system
    \begin{equation}
        \begin{aligned}
            \dot x^i(t)&=f^i(t,x^i(t),x^{i+1}(t)),&\quad& i\in \{0,\dots,N-1\},\\
            \dot x^N(t)&=f^N(t,x^N(t),x^0(t-\tau(t,x_t)))&\quad&
        \end{aligned}\label{eq:sdde}
    \end{equation}
    with state-dependent delay $\tau = \tau(t,\varphi)$ for $t\in\R$ and $\varphi \in \cC_\K$. 

As Walther pointed out in \cite{walther:02}, it seems not obvious how to single out a tractable class of SDDEs which contains a large set of examples that are well motivated. Yet, in what follows, we will give two frequently used classes of state-dependent delays for which the results of \cref{section:main-results} apply. We also mention that hypothesis \ref{delay-prop} can be verified for several other types of state-dependent delays not discussed here, see e.g.\ the scalar models treated in  \cite{arino:hadeler:hbid:98,hu:wu:10,magal:arino:00}. For a general -- but not so recent -- overview on SDDEs we refer the reader to \cite{hartung:krisztin:walther:wu:06}.

\subsection{Threshold-type delay}
	Threshold-type delays arise in many applications, see e.g.~\cite{balazs:getto:rost:21,gedeon:et-al:arxiv:2024,waltman:74,smith:83} and references therein.

        Consider the system \eqref{eq:sdde}, where $\tau\colon \CK\to \R$ is defined by
        \begin{align}
            \int_{-\tau(\varphi)}^0a(\varphi(s))\intd s=1,\label{eq:thr_delay}
        \end{align}
        with $a\colon\R\to[a_{\min},a_{\max}]$ continuous, $0<a_{\min}<a_{\max}$. Clearly, $\tau(\varphi)$ is well-defined, and bounded by $r\coloneqq {1}/{a_{\min}}$.

        Under suitable dissipativity conditions on $f$ and by restricting the phase space for $L_0$-Lipschitz continuous functions with some $L_0>0$, we may guarantee the existence of a unique solution that exists forward in time for all $t$ and initial function $\varphi$. We refer to \cite{bartha:garab:krisztin:25} for more details, where the scalar case was dealt with. For our purposes, it suffices to assume that we have a solution $x$ of system \eqref{eq:sdde}--\eqref{eq:thr_delay} on an interval $I$.

        Assume, as always, that hypothesis \ref{f-cont}--\ref{f-feedback} hold. For a fixed solution $x$ of system \eqref{eq:sdde}--\eqref{eq:thr_delay} on an interval $I$, define $\eta(t)=t-\tau(x_t)$. Let $k\in\N$ and $h>0$ be fixed. Then for any $t\in I$, if $t+h$ and $\eta^k (t)$ are also in $I$, then  we obtain
        \begin{equation*}
            \int_{\eta^k(t)}^ta(x^0(s))\intd s=k=\int_{\eta^k(t+h)}^{t+h}a(x^0(s))\intd s.
        \end{equation*}
        Subtracting $\int_{\eta^k(t+h)}^ta(x^0(s))\intd s$, we arrive at
        \[            \int_{\eta^k(t)}^{\eta^k(t+h)}a(x^0(s))\intd s=\int_t^{t+h}a(x^0(s))\intd s>0,\]
        hence $\eta^k$ is strictly increasing, furthermore, inequalities
       \begin{align*}\int_{\eta^k(t)}^{\eta^k(t+h)}a_{\min}\intd s&\le\int_t^{t+h}a_{\max}\intd s,\\
            a_{\min}|\eta^k(t+h)-\eta^k(t)|&\le a_{\max} h,\\
            |\eta^k(t+h)-\eta^k(t)|&\le \frac{a_{\max}}{a_{\min}}h,
        \end{align*}            
        hold, that is, $\eta^k$ is Lipschitz continuous with
           \[ \lip{\eta^k} \le\frac{a_{\max}}{a_{\min}},\]
        independently of $k$, so $\eta$ satisfies hypotheses \ref{delay-prop} and \ref{eta-iterates-Lip}, and $x$ solves an equation of the form \eqref{eq:x_unidir}. Hence, \cref{thm:V-decreasing,thm:V-dropping,thm:solutions-regularize,th:x_0} apply.
    
    \subsection{Implicit delay given by two values}
    Assume that $f$ is bounded and let
     $$L_0=\sup_{\substack{\xi\in \R^3\\ 0\leq i\leq N }}|f^i(\xi)|.$$
        % $$\adjustlimits\max_{i=0,\dots,N}\sup_{t\in\R}\max_{u,v\in[-M,M]}|f(t,u,v)|$$
    Consider equation \eqref{eq:sdde} on a modified spate-space
    \[\CKL \coloneqq \{\varphi \in \CK : \varphi|_{[-r,0]} \text{ is $L_0$-Lipschitz continuous}\},\]
   and let $\tau$ be given by 
    \begin{enumerate}[label=(ID$_\arabic*$)]
        \item 	\label{eq:delay:implicit-tau}
		 $\tau(t,x_t)= R(x(t), x^0(t - \tau),t)$,
        \item \label{eq:delay:implicit-cond} 
        $R \colon \R^{N+1}\times\R\times\R \to (0, r]$  is  Lipschitz continuous in all three variables with Lipschitz constants $\lip{R_1}$, $\lip{R_2}$, and $\lip{R_3}$, respectively. Moreover,
        \[\lip R_3 < 1\quad\text{and}\quad\frac{\lip R_1+2\lip R_2}{1-\lip R_3} < \frac{1}{L_0}.\]
    \end{enumerate}

    We will shortly see that $\tau(t,\varphi)$ is well-defined. Just as in the case of threshold delays, and analogous to the scalar case, treated in \cite{bartha:garab:krisztin:25}, additional assumptions can be given to establish the existence and uniqueness of solutions and the existence of a pullback (resp.\ global) attractor. 
    
    Such SDDEs have been studied, for example, in \cite{krisztin:arino:01,walther:09,bartha:krisztin:18,kennedy:19}.
	\begin{proposition}
		\label{prop:delay:twov:welldefined-lip}
		For any $\varphi \in \CKL$ and $t\in \R$, $\tau(t,\varphi)$ is 
		well-defined by \emph{\ref{eq:delay:implicit-tau}--\ref{eq:delay:implicit-cond}}, moreover, it is strictly positive and Lipschitz continuous in both variables. 
	\end{proposition}
	
	\begin{proof}
		For $\varphi \in \CKL$, $t\in\R$ and $s \in [0, r]$, 
		define
		$$\sigma(t,\varphi)(s) = R(\varphi(0),\dots,\varphi(N), \varphi(-s),t).$$
		Then the estimate
		\begin{equation*}
			\begin{split}
				\MoveEqLeft \left| \sigma(t,\varphi)(s_1) - \sigma(t,\varphi)(s_2) \right|\\
				&=
				\left| R(\varphi(0),\dots,\varphi(N), \varphi(-s_1),t) - R(\varphi(0),\dots,\varphi(N), \varphi(-s_2),t) \right| \\
				&\leq \lip{R_2} |\varphi(-s_1) - \varphi(-s_2)| \leq \lip{R_2} L_0 |s_1 - s_2|
			\end{split}
		\end{equation*}
		holds for all $s_1,s_2\in [0,r]$. Hence, by \ref{eq:delay:implicit-cond}, the map $\sigma(t,\varphi) \colon [0, r] \to [0, r]$ is a contraction for all $\varphi \in \CKL$ and $t\in\R$, and it has a unique fixed point denoted by $\tau(t,\varphi)$. Therefore, $\tau(t,\varphi)$ is well-defined and positive. 
		Let $t,s\in\R$, $\varphi, \psi \in \CKL$. Then, by the triangle inequality and the Lipschitz continuity of $R$,
		\begin{equation*}
			\begin{split}
				\MoveEqLeft[3]\left| \tau(t,\varphi) - \tau(s,\psi) \right|\\
				\leq{}&| R(\varphi(0),\dots,\varphi(N), \varphi(- \tau(t,\varphi)),t) - R(\psi(0),\dots,\psi(N), \psi(- \tau(s,\psi)),s) | \\
				\leq {} &| R(\varphi(0),\dots,\varphi(N), \varphi(- \tau(t,\varphi)),t) - 
				R(\varphi(0),\dots,\varphi(N), \varphi(- \tau(t,\varphi)),s) |  \\
				&{}+| R(\varphi(0),\dots,\varphi(N), \varphi(- \tau(t,\varphi)),s) - 
				R(\psi(0),\dots,\psi(N), \varphi(- \tau(t,\varphi)),s) |  \\
				&{}+| R(\psi(0),\dots,\psi(N), \varphi(- \tau(t,\varphi)),s) - R(\psi(0),\dots,\psi(N), \varphi(- \tau(s,\psi)),s) | \\ 
				&{}+| R(\psi(0),\dots,\psi(N), \varphi(- \tau(s,\psi)),s) - R(\psi(0),\dots,\psi(N), \psi(- \tau(s,\psi)),s) | \\
				\leq {} &\lip{R_3}|t-s|+\lip{R_1} \|\varphi - \psi\|\\
				&{}+\lip{R_2} | \varphi(- \tau(t,\varphi)) - \varphi(- \tau(s,\psi)) |  +\lip{R_2} \|\varphi - \psi\| \\
				\leq {} &\lip{R_3}|t-s|+\left( \lip{R_1} + \lip{R_2} \right) \|\varphi - \psi\|+\lip{R_2} L_0 | \tau(t,\varphi) - \tau(s,\psi) |
			\end{split}
		\end{equation*}
		follows.
		
		Thus, using \ref{eq:delay:implicit-cond} again, we obtain that 
		$\tau$ is Lipschitz continuous on $\R\times\CKL$ with 
		\begin{equation*}%\label{lip:const:twov}
			\lip{\tau_1} \leq \frac{\lip{R_3}}{1 - \lip{R_2} L_0},\qquad \lip{\tau_2} \leq \frac{\lip{R_1} + \lip{R_2}}{1 - \lip{R_2} L_0}. \qedhere
		\end{equation*}
	\end{proof}
	
	Next we prove that for any fixed solution $x$ of \eqref{eq:sdde}, \ref{eq:delay:implicit-tau}--\ref{eq:delay:implicit-cond},  $\eta( t)=t-\tau(t,x_t)$ is strictly increasing in $t$.
	
	\begin{proposition}
		\label{prop:delay:twov:eta}
		Suppose that hypotheses \emph{\ref{f-cont}--\ref{f-feedback}} hold and let $x$ be a solution of \eqref{eq:sdde}, \emph{\ref{eq:delay:implicit-tau}--\ref{eq:delay:implicit-cond}} on an interval $I$. Then $\eta(t)\coloneqq  t-\tau(t,x_t)$ is strictly increasing on $I$.
	\end{proposition}
	
	\begin{proof}
    		Let $t, s \in I$ such that $t > s$. From  \ref{eq:delay:implicit-cond} we readily obtain
            \[\frac{\lip{R_3}+(\lip{R_1} + \lip{R_2})L_0}{1 - \lip{R_2} L_0}< 1,\]
            from which
		\begin{align*}
			\left| \tau(t,x_t) - \tau(s,x_s) \right| &\leq \lip{\tau_1}|t-s|+\lip{\tau_2} \|x_t - x_s\|\\
            &\leq 
			\frac{\lip{R_3}+(\lip{R_1} + \lip{R_2})L_0}{1 - \lip{R_2} L_0}|t - s|<|t-s|
		\end{align*}
            follows. Hence
		\begin{equation*}
			\eta( t) - \eta( s) = (t - \tau(t,x_t)) - (s - \tau(s,x_s)) = (t - s) - ( \tau(t,x_t) - \tau(s,x_s) ) > 0. \qedhere
		\end{equation*}
	\end{proof}

    In view of \cref{prop:delay:twov:eta,prop:delay:twov:welldefined-lip}, the results of \cref{sec:V-basic-prop} apply to equation \eqref{eq:sdde}, \ref{eq:delay:implicit-tau}--\ref{eq:delay:implicit-cond}.

    As for finiteness results on $V$, first note that \ref{eta-iterates-Lip} may be hard to guarantee, in general. Nevertheless, interestingly, one can still prove using a similar argument as seen at \cref{th:x_0} that the statements of the theorem remain true for solutions of \eqref{eq:sdde}, \ref{eq:delay:implicit-tau}--\ref{eq:delay:implicit-cond}, as well, at least when $R$ does not explicitly depend on $t$. This is formulated in the next proposition.
    
	\begin{proposition}\label{th:x_0_r}
		Assume that hypotheses \emph{\ref{f-cont}--\ref{f-feedback}}  and \emph{\ref{f-lin-bounded}} hold and that $R$ does not depend on its third argument. Let $x$ be a backwards-bounded entire solution of \eqref{eq:sdde}, \emph{\ref{eq:delay:implicit-tau}--\ref{eq:delay:implicit-cond}},  $\tau(t)\coloneqq \tau(t,x_t)$, and  $\eta(t)\coloneqq t-\tau(t)$. Then the following statements hold.
        \begin{enumerate}[label=(\roman*)]
        \item If there exists $\sigma\in\R$ such that $x(\eta^k(\sigma))=0\in \R^{N+1}$ for all $k\in\N_0$, then $x(t)=0$ for all $t\in\R$. 
        \item If $x$ is nontrivial, then $V(x_t,-\tau(t))$ is finite for $t\in \R$.
        \end{enumerate}
	\end{proposition}
	
	\begin{proof}
		(i) Define $z_k^i(t)=x(\eta^k(\sigma)+t)$ for $t\in [0,r]$ and $$M\coloneqq \sup\{|x^i(t)| : t \in (-\infty,r],\ 0\leq i\leq N\}.$$ Observe the different definition of $z_k$ compared to that in the proof of \cref{th:x_0}. Note that $z_k(0)=0$ for all $k\in \N_0$. By virtue of equation \eqref{eq:sdde}, hypothesis \ref{f-lin-bounded} and the backward boundedness of $x$ we have
		\begin{align}
			\left|\dot z_k^i(t)\right|&=\left|\dot x^i(\eta^k(\sigma)+t)\right|=\left|f^i(t,x^i(\eta^k(\sigma)+t),x^{i+1}(\eta^k(\sigma)+t))\right|\nonumber\\
			&=\left|f^i(t,z_k^i(t),z_k^{i+1}(t))\right|\le C\left(\left|z_k^i(t)\right|+\left|z_k^{i+1}(t)\right|\right)\nonumber
			\intertext{for $i=0,\dots,N-1$, and}
			\left|\dot z_k^N(t)\right|&=\left|\dot x^N(\eta^k(\sigma)+t)\right|=\left|f^N(t,x^N(\eta^k(\sigma)+t),x^0(\eta(\eta^k(\sigma)+t)))\right|\nonumber\\
			&=\left|f^N(t,z_k^N(t),x^0(\eta(\eta^k(\sigma)+t)))\right|\le C\left(\left|z_k^N(t)\right|+\left|x^0(\eta(\eta^k(\sigma)+t))\right|\right).\label{eq:z_bound}
		\end{align}
		We can estimate the last term of \eqref{eq:z_bound} in the following way:
		\begin{equation*}
			\begin{split}
				|x^0(\eta(\eta^k(\sigma)+t))|&\leq   \bigl|x^0(\eta(\eta^k(\sigma)+t)) - x^0(\eta^{k+1}(\sigma) + t)\bigr|  + |x^0(\eta^{k+1}(\sigma) + t)\bigr|\\
				&\leq  L_0 \bigl|t + \eta^k(\sigma) - \tau(\eta^k(\sigma)+t) - t - \eta^k(\sigma) + \tau(\eta^k(\sigma))\bigr| + |z_{k+1}^0(t)| \\
				&=   L_0 |\tau(\eta^k(\sigma)) - \tau(\eta^k(\sigma)+t)| + |z^0_{k+1}(t)| .
			\end{split}
		\end{equation*}
		From the equations
		\begin{align*}
			\tau(\eta^k(\sigma))&=R\bigl(x(\eta^k(\sigma)),x^0(\eta^{k+1}(\sigma))\bigr)\\
			\shortintertext{and}
			\tau(\eta^k(\sigma)+t)&=R\bigl(x(\eta^k(\sigma)+t),x^0(\eta(\eta^k(\sigma)+t))\bigr),
		\end{align*}
		we conclude that
		\begin{equation*}
			\begin{aligned}
				|x^0(\eta(\eta^k(\sigma)+t)) | 
				&\leq  L_0 \lip{R_1} \bigl|x^0(\eta^{k}(\sigma) + t) - x^0(\eta^k(\sigma))\bigr| \\
				&\mathrel{\hphantom{=}}{}+  L_0 \lip{R_2} \bigl|x^0(\eta(\eta^k(\sigma)+t)) - x^0(\eta^{k+1}(\sigma))\bigr| + |z^0_{k+1}(t)| \\
				&=  L_0 \lip{R_1}|z^0_k(t)| +  L_0 \lip{R_2} |x^0(\eta(\eta^k(\sigma)+t))| + |z^0_{k+1}(t)|,
			\end{aligned}
		\end{equation*}
		or equivalently,
		\[
		|x^0(\eta(\eta^k(\sigma)+t))| \leq \frac{\lip{R_1}L_0 }{1 - \lip{R_2} L_0}|z^0_k(t)| + \frac{1}{1 - \lip{R_2} L_0}|z^0_{k+1}(t)|.
		\]
		Substituting this into \eqref{eq:z_bound} yields
		\begin{equation*}
			|\dot{z}^N_k(t)| \leq \left(C  + \frac{C L_0 \lip{R_1}}{1 - \lip{R_2} L_0}\right)|z_k(t)| + \frac{C}{1 - \lip{R_2} L_0} |z_{k+1}(t)|.
		\end{equation*}
		By setting \[C_1 \coloneqq C + \frac{C L_0 \lip{R_1}}{1-\lip{R_2} L_0} + \frac{C}{1-\lip{R_2} L_0}, \] we obtain that inequality
		\[|\dot{z}_k(t)| \leq C_1( |z_k(t)| +  |z_{k+1}(t)|)\]
		holds for all $t\in \R $ and $k\in \N_0$.
		\medskip
		
		From this point we can continue with a similar argument as in the proof of  \cref{th:x_0}. Recall that $z^k(0)=0$. Thus, setting  $C_2\coloneqq 2 M(1+C_1)$, we obtain that
		$|z_k(t)| \leq C_2$ and  $|\dot{z}_k(t)| \leq C_2$ hold for all $t \in \R $ and $k \in \N$. Hence
		\begin{equation}
			\label{lipsch2}
			|z_k(t)| \leq C_2 \int^t_0\left(|z_k(s)| + |z_{k+1}(s)|\right)\intd s\quad  \text{for all } t \in [0,r].
		\end{equation}
		
		Let $Z(t)=(z_k(t))_{k=0}^\infty \in l^\infty$. The uniform boundedness of the components and their derivatives give the continuity of $Z$. From  inequality \eqref{lipsch2} we conclude that
		\begin{equation}
			\label{lipsch3}
			\|Z(t)\|_\infty \leq 2 C_2\int^t_0 {\|Z(s)\|_\infty \, \intd s}
		\end{equation}
		holds for all $t \in [0,r]$. 
		
		Applying the Gronwall--Bellman lemma yields that $\|Z(t)\|_\infty \equiv 0$, that is, $z_k(t) = 0$ for all $k \in \N$ and $ t\in [0,r]$, which together with \cref{th:x_t}\,\ref{0-is-invariant} yield $x(t)=0$ for all $t\in \R$.
    \medskip

    (ii) In view of statement (i), the proof of \cref{th:x_0}\,\ref{statement:V-finite} holds here.
    \end{proof}

\section{Concluding remarks}
In this work we generalized the results of Mallet-Paret and Sell \cite{mallet-paret:sell:96-lyapunov} on the existence and properties of a discrete Lyapunov function for unidirectional cyclic systems for the case of time-variable or state-dependent delay. We conclude this paper by discussing  possibilities to further generalize the results and  directions of future works.

\subsubsection*{Possible generalizations} 
It is worth noticing that equation \eqref{eq:x_bidir} was more general in two ways -- apart from allowing only constant delays, of course. On the one hand, it also admitted a bidirectional structure, moreover, more delays were also allowed -- one delay in each equation, i.e.\ $x^{i}(t-\tau_{i})$ in the $(i-1)$-th equation, in place of $x^i(t)$. The constant delay and cyclic structure allowed one to transform the equation into the normalized form \eqref{eq:x_bidir}. We conjecture that our results can be further generalized to the case when the delays $\tau_i$ depend on time, i.e. to systems of the form
\begin{equation}
\begin{alignedat}{3}
    \dot x^0(t)&=f^0(t,x^0(t),x^1(t-\tau_1(t))),&\qquad &\\
    \dot x^{i}(t)&=f^{i}(t,x^{i-1}(t),x^i(t),x^{i+1}(t-\tau_{i+1}(t))),&\qquad &i=1,\dots,N,
\end{alignedat}\label{eq:x_bidir-time-variable}
\end{equation}
perhaps at the cost of additional assumptions on the delays. We see two possible approaches to tackle this problem. 1) If the functions $\tau_i$ are continuously differentiable, then a time transformation similar to the one introduced by Cao \cite{cao:92} in the scalar case could be combined with the one applied in \cite{mallet-paret:sell:96-lyapunov} to obtain a system of equations with a single constant delay and of the normalized form \eqref{eq:x_bidir}. Then the results of the latter paper could be applied directly to the transformed system. This approach also has its limitations. For example, if one has a system of SDDEs with multiple delays, then each solution of the SDDE system will be transformed into a different normalized equation, so it will be very hard to get any useful information on solutions of the former equation from results obtained for the latter. The differentiability of delays $\tau_i(x_t)$ with respect to $t$ could also be violated in certain SDDEs. 2) A natural way to avoid both the extra condition on smoothness of the delays and the issue of losing information during the transformation could be to develop a (time-dependent) discrete Lyapunov function directly for the system \eqref{eq:x_bidir-time-variable}, on the phase space of $C([-r,0],\R^{N+1})$. However, such a theory is not yet available, even in the case of a constant delay. 

We believe that the assumption on the continuity of $\eta$ is not essential for \cref{thm:V-decreasing} and could be replaced by piecewise continuity with some extra effort in the proof. This could be relevant in models where the feedback is based on observed data on the state, which is not continuously monitored, but is measured, for instance, $T$-periodically. Then we could end up with something like $\eta(t)=\lfloor t/T\rfloor T-\tau_0$, where $\lfloor{}\cdot{}\rfloor$ denotes the floor function and $\tau_0>0$ reflects an extra reaction lag.

\subsubsection*{Potential applications}
The theory developed in this work can potentially have applications in various topics, such as proving the existence of slowly (or rapid) oscillatory periodic solutions, studying the geometric structure of the unstable manifold of the origin, or developing a Poincaré--Bendixson-type theory for certain cyclic SDDE systems. Our next goal is to prove the existence of a Morse decomposition of the global attractor for autonomous cyclic SDDE systems with delays treated in \cref{sec:SDDEs}. A Morse decomposition means, roughly speaking, that the dynamics is gradient-like on the global attractor. Such a result has been established very recently by Bartha, Krisztin and one of the authors for scalar equations and negative feedback \cite{bartha:garab:krisztin:25}. We hope to extend these results to the case of $N>0$ by combining our current findings with the ideas and techniques applied in \cite{bartha:garab:krisztin:25} and \cite{garab:20}.

\section*{Acknowledgement}
This research was supported by the National Research, Development and Innovation (NRDI) Fund, Hungary, [project no.\ TKP2021-NVA-09 and FK~142891], and by the National Laboratory for Health Security [RRF-2.3.1-21-2022-00006].
Á.~G.\ was supported by the János Bolyai Research Scholarship of the Hungarian Academy of Sciences.

\printbibliography

\end{document}